\documentclass[11pt]{article}
\usepackage{amsfonts}
\usepackage{color}
\usepackage{multicol}
\usepackage{amssymb}
\usepackage{graphicx}
\usepackage{epsfig}
\usepackage{epstopdf}
\usepackage{amsmath,amssymb}
\usepackage{bm}
\usepackage{graphicx}
\usepackage{epsf,amsmath}
\usepackage{amsthm}
\usepackage{graphicx,epsfig}

\renewcommand{\baselinestretch}{1.6}
\def\singlespace{\def\baselinestretch{1}\@normalsize}
\renewcommand{\baselinestretch}{1.36}
\newtheorem{thm}{Theorem}[section]

\newtheorem{lem}{Lemma}[section]

\newtheorem{definition}{Definition}[section]
\newtheorem{rem}{Remark}[section]

\newtheorem{assumption}{Assumption}

\def\E{{\rm E}}

\def \tr{\rm tr}
\def \and{\rm and}

\def \RCV{\rm RCV}
\def \TVA{\rm TVA}

\newcommand{\bu}{\mbox{\bf u}}

\newcommand{\bz}{\mbox{\bf z}}

\newcommand{\bI}{\mbox{\bf I}}

\newcommand{\bM}{\mbox{\bf M}}

\newcommand{\bS}{\mbox{\bf S}}
\newcommand{\bU}{\mbox{\bf U}}
\newcommand{\bV}{\mbox{\bf V}}
\newcommand{\bW}{\mbox{\bf W}}
\newcommand{\bX}{\mbox{\bf X}}

\newcommand{\bSig}{\mbox{\boldmath $\Sigma$}}

\newcommand{\1}{\bm{1}}

\newcommand{\cS}{\mathcal{S}}

\renewcommand{\(}{\left(}
\renewcommand{\)}{\right)}

\newcommand{\De}{\Delta}

\newcommand{\ga}{\gamma}

\newcommand{\bLa}{\pmb\Lambda}

\def\eqd{\,{\buildrel d \over =}\,}

\begin{document}

\title{{\Large Shrinkage Estimation of Covariance Matrix for Portfolio
Choice with High Frequency Data}}
\author{ \textsc{Cheng Liu, Ningning Xia and Jun Yu} \thanks{%
Liu is an assistant professor in Economics and Management School of Wuhan
University, Hubei, China. Email: chengliu\_eco@whu.edu.cn. Xia is an
assistant professor in School of Statistics and Management, Shanghai
University of Finance and Economics. Email: xia.ningning@mail.shufe.edu.cn.
Yu is a professor in School of Economics and Lee Kong Chian School of
Business, Singapore Management University. Email: yujun@smu.edu.sg.} }
\maketitle

\begin{abstract}
This paper examines the usefulness of high frequency data in estimating the
covariance matrix for portfolio choice when the portfolio size is large. A
computationally convenient nonlinear shrinkage estimator for the integrated
covariance (ICV) matrix of financial assets is developed in two steps. The
eigenvectors of the ICV are first constructed from a designed time variation
adjusted realized covariance matrix of noise-free log-returns of relatively
low frequency data. Then the regularized eigenvalues of the ICV are
estimated by quasi-maximum likelihood based on high frequency data. The
estimator is always positive definite and its inverse is the estimator of
the inverse of ICV. It minimizes the limit of the out-of-sample variance of
portfolio returns within the class of rotation-equivalent estimators. It
works when the number of underlying assets is larger than the number of time
series observations in each asset and when the asset price follows a general
stochastic process. Our theoretical results are derived under the assumption
that the number of assets ($p$) and the sample size ($n$) satisfy $%
p/n\rightarrow y>0$ as $n\rightarrow \infty $. The advantages of our
proposed estimator are demonstrated using real data.
\end{abstract}

\noindent {\small \textit{Some key words}: Portfolio Choice, High Frequency
Data; Integrated Covariance Matrix; Shrinkage Function.}


\newpage

\section{Introduction}

\label{introduction}The portfolio choice problem has been an important topic
in modern financial economics ever since the pioneer contribution by
Markowitz (1952). It is well-known in the literature that constructing an
optimal portfolio requires a good estimate for the second moment of the
future return distribution, i.e., the covariance matrix of the future
returns. The simplest situation for estimating the covariance matrix is when
the returns are independent and identically normally distributed (IID) over
time. In this case, the maximum likelihood estimator (MLE) is the sample
covariance matrix and the efficiency of MLE is justified asymptotically.

However, there are at least two problems for using the sample covariance
matrix to select the optimal portfolio in practice. First, when the
portfolio size is large, the sample covariance matrix is found to lead to
poor performances in the selected portfolio; see Jobson and Korkie (1980)
and Michaud (1989). Not surprisingly, the sample covariance matrix is rarely
used by practitioners when the portfolio size is large. The reason for the
poor performances is due to the degree-of-freedom argument. That is, too many
parameters have to be estimated in the covariance matrix when the portfolio
size is large. In fact, if the portfolio size is larger than the number of
time series observations in each asset, the sample covariance is always
singular. Second, the returns are not IID over time. This is because
typically the covariance is time varying. In this case, the asymptotic
justification for using the sample covariance matrix is lost.

Many alternative estimators of the large dimensional covariance matrix for
portfolio choice have been proposed in the literature. A rather incomplete
list includes Ledoit and Wolf (2003, 2004, 2014), Frahm and Memmel (2010),
DeMiguel, Garlappi, and Uppal (2009), DeMiguel, Garlappi, Nogales, and Uppal
(2009), Kan and Zhou (2007), Fan, Fan and Lv (2008), Pesaran and Zaffaroni
(2009), Tu and Zhou (2011). Most studies use dimension reduction
techniques. One of the techniques uses factor (either observed factors or
latent factors) models. Another approach uses a statistical technique known
as shrinkage, a method first introduced by Stein (1956). Murihead (1987)
reviewed the literature on shrinkage estimators of the covariance matrix.
All these estimators are constructed from low frequency data (daily, weekly
or monthly data) over a long period (one year or more). However, if the
investment period of a portfolio is much shorter (say one day or one week or
one month) which is empirically more relevant, given the time varying nature
of the covariance, we expect the covariance in the near future to be similar
to the average covariance over an immediate recent time period but not to
that over a long time period. Hence, even if data over a long time period is
available, one may only prefer using data over a short period. If low
frequency data over a short time period are used, however, the
degree-of-freedom argument will be applicable.

The recent availability of quality high-frequency data on financial assets
has motivated a growing literature devoted to the model-free measurement of
covariances. In a recent study, Fan, Li and Yu (2012) proposed to use
high-frequency data to estimate the ICV over a short time period for the
purpose of portfolio choice. Their setup allows one to impose gross exposure
constraints. The use of gross exposure constraints plays a similar role to
the no-short-sale constraint in Jagannathan and Ma (2003). Fan, Li and Yu
(2012) demonstrated the substantial advantages of using high-frequency date
in both simulation and empirical studies.

There are several reasons why it is better to use high frequency data to
estimate the covariance matrix. First, the use of high frequency data
drastically increases the sample size. This is especially true for liquid
assets. Second, one does not need to assume returns are IID any more for
establishing the large sample theory for the estimator. This generalization
is important due to the time-varying nature of spot covariance. Not
surprisingly, the literature on estimating the ICV based on high frequency
data is growing rapidly.

In this paper, we also use high frequency data to estimate the ICV for the
purpose of portfolio choice. Unlike Fan, Li and Yu (2012) where portfolio
choice is done under pre-specified exposure constraints, we focus our
attention on how to get a good shrinkage estimator of the ICV without any
pre-specified constraint.\footnote{%
DeMiguel, Garlappi, Nogales, and Uppal (2009) showed that adding a
constraint for 1-norm of weights is equivalent to shrinkage the estimator of
covariance matrix.} This shift of focus is due to the lack of guidance on
how to specify the gross exposure constraints. Our estimator designs the
shrinkage function as in Ledoit and Wolf (2014). However, we differ from
Ledoit and Wolf (2014) in the following important ways. First, instead of
applying the shrinkage function to the eigenvalues of sample covariance
matrix by assuming the returns are IID, we regularize the eigenvalues of a
designed time variation adjusted (TVA)\textbf{\ }realized covariance matrix
under the assumption that the covariance matrix is time varying. Second,
instead of using low frequency data, we use high frequency data for
constructing the designed TVA realized covariance matrix and estimating its
regularized eigenvalues. We show that our proposed estimator, which will be given in Section 3,  not only has
some desirable properties in terms of estimating the ICV, but also
asymptotically achieves the minimum out-of-sample portfolio risk.


The paper is organized as follows. In Section 2 we set up the portfolio
choice problem. Section 3 introduces our estimator and discusses its
properties and implementations. In Section 4, we compare the out-of-sample
performance of our proposed method with several methods proposed in the
literature using actual data, including the equal weight, the linear
shrinkage estimator of Ledoit and Wolf (2004), and the high frequency method
of Fan, Li and Yu (2012). Section 5 concludes. The appendix collects the
proof of our theoretical results.

\section{Portfolio Selection: The Setup}

\label{subsec:intro} 

Suppose that a portfolio is constructed based on a pool of $p$ assets whose
log-price is denoted by $\mbox{\bf X}_{t}=(X_{1t},\cdots ,X_{pt})^{\prime }$%
, where $\mbox{\bf M}^{\prime }$ denotes the transpose of the vector or
matrix $\mbox{\bf M}$. Instead of assuming $\mbox{\bf X}_{t}$ follows a
Brownian motion which means that the log-returns are IID, we assume $%
\mbox{\bf X}_{t}$ follows a more general diffusion process as
\begin{equation}
d\mbox{\bf X}_{t}=\mbox{\boldmath $\mu$}_{t}dt+\pmb\Theta _{t}d\mbox{\bf B}%
_{t},  \label{eqref}
\end{equation}%
where $\mbox{\boldmath $\mu$}_{t}=(\mu _{1t},\cdots ,\mu _{pt})^{\prime }$
is a $p$-dimensional drift process at time $t$, $\pmb\Theta _{t}$ is a $%
p\times p$ (spot) covolatility matrix at time $t$, and $\mbox{\bf B}_{t}$ is
a $p$-dimensional standard Brownian motion.


A portfolio is constructed based on $\mbox{\bf X}_{t}$ with weight $%
\mbox{\bf
	w}_{T}$ which satisfies $\mbox{\bf w}_{T}^{\prime }\bm{1}=1$ at time $T$
and a holding period $\tau $, where $\bm{1}$ is a $p$-dimensional vector
with all elements being 1. Over the period $[T,T+\tau ]$, it has a return $%
\mbox{\bf w}_{T}^{\prime }\int_{T}^{T+\tau }d\mbox{\bf X}_{t}$, and has a
risk (variance)
\begin{equation*}
R_{T,T+\tau }(\mbox{\bf w}_{T})=\mbox{\bf w}_{T}^{\prime }\widetilde{%
\mbox{\boldmath
		$\Sigma$}}_{T,T+\tau }\mbox{\bf w}_{T},~~~\mathrm{where}~~\widetilde{%
\mbox{\boldmath
		$\Sigma$}}_{T,T+\tau }=\int_{T}^{T+\tau }\mathrm{E}_{T}%
\mbox{\boldmath
	$\Sigma$}_{t}dt,
\end{equation*}%
with $\mbox{\boldmath $\Sigma$}_{t}=\pmb\Theta _{t}\pmb\Theta _{t}^{\prime }$
being the (spot) covariance matrix at time $t$ and $\mathrm{E}_{T}$ denotes
the expectation conditional on information up to time $T$ (see Fan, Li and
Yu, 2012). Typically, the holding period $\tau $ is short (say one day or
one week or one month).


To focus on finding a good approximation for $\widetilde{%
\mbox{\boldmath
$\Sigma$}}_{_{T,T+\tau }}$, we consider the following global minimum
variance (GMV) problem:
\begin{equation}
\min_{\mbox{\bf w}_T}\mbox{\bf w}_{T}^{\prime }\widetilde{
\mbox{\boldmath
		$\Sigma$}}_{T,T+\tau }\mbox{\bf w}_{T}~~~~\mathrm{with}~~~~\mbox{\bf w}%
_{T}^{\prime }\bm{1}=1.  \label{gmv}
\end{equation}%
By taking the derivative of $\mbox{\bf w}_{T}$, we have the following
theoretical optimal weight,
\begin{equation}
\mbox{\bf w}_{T}=\dfrac{\widetilde{\mbox{\boldmath $\Sigma$}}_{T,T+\tau
}^{-1}\bm{1}}{\bm{1}^{\prime }\widetilde{\mbox{\boldmath $\Sigma$}}
_{T,T+\tau }^{-1}\bm{1}},  \label{min_risk_optimal_weight}
\end{equation}%
which is a function of the expected ICV conditional on the current time $T$,
i.e., $\widetilde{\mbox{\boldmath $\Sigma$}}_{T,T+\tau }$.

Denote the ICV over the period $[T-h,T]$ by
\begin{equation*}
{\mbox{\boldmath $\Sigma$}}_{T-h,T}:=\int_{T-h}^{T}\mbox{\boldmath $\Sigma$}%
_{t}dt.
\end{equation*}%
\newline
If $h$ is small, following Fan, Li and Yu (2012), we use the following
approximation
\begin{equation}  \label{approximate expected ICV}
\widetilde{\mbox{\boldmath $\Sigma$}}_{T,T+\tau }\approx \frac{\tau}{h}{\ %
\mbox{\boldmath $\Sigma$}}_{T-h, T}.
\end{equation}%
Consequently, the theoretical optimal weight becomes
\begin{equation}
\mbox{\bf w}_{T}=\dfrac{{\mbox{\boldmath $\Sigma$}}_{T-h, T}^{-1}\bm{1}}{ %
\bm{1}^{\prime }{\mbox{\boldmath $\Sigma$}}_{T-h, T}^{-1}\bm{1}}.
\label{min_risk_optimal_weightB}
\end{equation}

The reason for choosing a small $h$ from the historical sample (i.e. a small
time span for $[T-h,T]$) to approximate the expected ICV is due to the time
varying and persistent nature of the covariance matrix. If a big $h$ (say 10
years) is used and an average covariance matrix is used to approximate the
expected ICV, the approximation errors would be inevitably large. In fact,
as rightly argued in Fan, Li and Yu (2012), even when the true covariance
matrices are available, an average of them will still lead to large
approximation errors.

Let $\widehat{\mbox{\boldmath $\Sigma$}}_{T-h,T}^{\ast }$ denote a generic
(invertible) estimator of the ICV ${\mbox{\boldmath $\Sigma$}}_{T-h,T}$. The
plug-in estimator of the optimal portfolio weight for $\mbox{\bf w}_{T}$ in (%
\ref{min_risk_optimal_weightB}) is
\begin{equation*}
\hat{\mbox{\bf w}}_{T}^{\ast }:=\dfrac{\left( \widehat{%
\mbox{\boldmath
$\Sigma$}}_{T-h,T}^{\ast }\right) ^{-1}\bm{1}}{\bm{1}^{\prime }\left(
\widehat{\mbox{\boldmath $\Sigma$}}_{T-h,T}^{\ast }\right) ^{-1}\bm{1}}.
\end{equation*}%
We need to find the optimal $\widehat{\mbox{\boldmath $\Sigma$}}%
_{T-h,T}^{\ast }$ for portfolio choice. Given that the optimal portfolio is
typically meant to perform the best out-of-sample, following Ledoit and Wolf
(2014), we define a loss function for portfolio selection to be the
out-of-sample variance of portfolio returns conditional on $\widehat{%
\mbox{\boldmath $\Sigma$}}_{T-h,T}^{\ast }$,
\begin{equation}
\mathcal{L}(\widehat{\mbox{\boldmath $\Sigma$}}_{T-h,T}^{\ast },%
\mbox{\boldmath $\Sigma$}_{T-h,T})=\left( \hat{\mbox{\bf w}}_{T}^{\ast
}\right) ^{\prime }\mbox{\boldmath $\Sigma$}_{T-h,T}\hat{\mbox{\bf w}}%
_{T}^{\ast }=\dfrac{\bm{1}^{\prime }\left( \widehat{\mbox{\boldmath $\Sigma$}%
}_{T-h,T}^{\ast }\right) ^{-1}\mbox{\boldmath $\Sigma$}_{T-h,T}\left(
\widehat{\mbox{\boldmath $\Sigma$}}_{T-h,T}^{\ast }\right) ^{-1}\bm{1}}{%
\left\{ \bm{1}^{\prime }\left( \widehat{\mbox{\boldmath $\Sigma$}}%
_{T-h,T}^{\ast }\right) ^{-1}\bm{1}\right\} ^{2}},  \label{loss}
\end{equation}%
where we approximate $\widetilde{\mbox{\boldmath $\Sigma$}}_{T,T+\tau }$ by $\frac{\tau}{h}\bSig_{T-h, T}$ and ignore the scale $\frac{\tau}{h}$  without any loss.  The best estimator of the ICV is therefore the one that minimizes the loss
function $\mathcal{L}(\widehat{\mbox{\boldmath $\Sigma$}}_{T-h,T}^{\ast },%
\mbox{\boldmath $\Sigma$}_{T-h,T})$.

Although this paper mainly focuses on the GMV problem, our estimation
technique has a much wider implications for other problems that also require
the estimation of ICV, including the Markowitz portfolios with and without
estimating the conditional mean. In the empirical studies, we will show the
usefulness of our proposed method in the context of the Markowitz portfolio.

\section{The New Estimator of ICV}

Denote the trading time points for the $i$th asset by $0\leq
t_{i1}<t_{i2}<...<t_{i,N_{i}}\leq T$ with $i=1,...,p$. It is difficult to
estimate the ICV based on tick-by-tick high frequency data when the number
of stocks ($p$) is large for the following reasons. First, data are always
non-synchronous. Second, data are contaminated by microstructure noises.
Denote $Y_{i,t_{ij}}$ the log-price of the $i$th asset at time $t_{ij}$ and $%
X_{i,t_{ij}}$ the latent log efficient price of the $i$th asset. Then
\begin{equation*}
Y_{i,t_{ij}}=X_{i,t_{ij}}+\epsilon _{i,t_{ij}},
\end{equation*}%
where $\epsilon _{i,t_{ij}}$ is the market microstructure noise at time $%
t_{ij}$. Third,  the spot covariance matrix $\mbox{\boldmath $\Sigma$}_{t}$ of returns of latent log-price $%
\mbox{\bf X}_{t}$  is time varying.
Fourth, the returns of the efficient price are not independent over time. To find a good estimator for the ICV, we first  introduce an initial estimator, denoted the time variation adjust (TVA) realized covariance matrix, and discuss its disadvantages for estimating the ICV in subsection \ref{initial estimator}. To improve the initial estimator, we propose to regularize its eigenvalues.  In subsection \ref{Theoretical background for regularizing}, we provide the theoretical background for regularizing the eigenvalues of  TVA realized covariance matrix. We then demonstrate how to regularize  its eigenvalues  in subsection \ref{Regularized estimators of eigenvalues}.

\subsection{The initial estimator of ICV: TVA}\label{initial estimator}

To simplify the problem, we propose the following structural assumption for $%
\mbox{\bf X}_{t}$. The same assumption was also used in Zheng and Li (2011).

\begin{definition}
(Class $\mathcal{C}$). \label{class_C} Suppose that $\mbox{\bf X}_{t}$ is a $%
p$-dimensional process satisfying Equation \eqref{eqref}. We say that $%
\mbox{\bf
		X}_{t}$ belongs to class $\mathcal{C}$ if, almost surely, there exist $%
\gamma _{t}\in D([T-h,T];\mathbb{R})$ and $\pmb\Lambda $ a $p\times p$
matrix satisfying ${\mathrm{tr}}(\pmb\Lambda \pmb\Lambda ^{\prime })=p$ such
that
\begin{equation*}
\pmb\Theta _{t}=\gamma _{t}\pmb\Lambda ,
\end{equation*}%
where $D([T-h,T];\mathbb{R})$ stands for the space of c$\grave{a}$dl$\grave{a%
}$g functions from $[T-h,T]$ to $\mathbb{R}$.
\end{definition}

\begin{rem}
Class $\mathcal{C}$ allows the covariance matrix to be time varying because $%
\gamma _{t}$ is time varying. The assumption of $\pmb\Theta _{t}=\gamma _{t}%
\pmb\Lambda $ may be too strong than necessary but facilitates the
mathematical proof of the results in the present paper.
\end{rem}

If $\mbox{\bf X}_{t}$ belongs to class $\mathcal{C}$, we can decompose
\begin{equation*}
\mbox{\boldmath $\Sigma$}_{T-h,T}=\int_{T-h}^{T}\gamma _{t}^{2}dt\cdot \pmb%
\Lambda \pmb\Lambda ^{\prime }=\mbox{\bf P}\left( \int_{T-h}^{T}\gamma
_{t}^{2}dt\cdot \bm{\Gamma}\right) \mbox{\bf P}^{\prime },
\end{equation*}%
where $\bm\Gamma $ is a diagonal matrix, $\mbox{\bf P}$ an orthogonal
matrix, and $\mbox{\bf P}\bm\Gamma \mbox{\bf P}^{\prime }$ the
eigen-decomposition of $\pmb\Lambda \pmb\Lambda ^{\prime }$ such that the
eigenvalues and eigenvectors of $\mbox{\boldmath $\Sigma$}_{t}=\pmb\Theta
_{t}\pmb\Theta _{t}^{\prime }$ are time varying and invariant respectively.

To estimate $\mbox{\boldmath $\Sigma$}_{T-h,T}$, Zheng and Li (2011)
proposed to use the so-called TVA realized covariance matrix over the period
$[T-h,T]$, which is defined as
\begin{equation}
\mbox{\bf S}_{T-h,T}^{\mathrm{TVA}}=\dfrac{{\mathrm{tr}}\left( {%
\sum_{k=1}^{n}\Delta \mbox{\bf X}_{k}\Delta \mbox{\bf X}_{k}^{\prime }}%
\right) }{p}\cdot \breve{\mbox{\bf S}}_{T-h,T},~~\mathrm{where}~~\breve{%
\mbox{\bf S}}_{T-h,T}=\dfrac{p}{n}\sum_{k=1}^{n}\dfrac{\Delta \mbox{\bf X}_{{%
k}}\Delta \mbox{\bf X}_{{k}}^{\prime }}{|\Delta \mbox{\bf X}_{k}|^{2}},
\label{S_TVA}
\end{equation}%
$\Delta \mbox{\bf X}_{k}=\mbox{\bf X}_{\tau _{k}}-\mbox{\bf X}_{\tau _{k-1}}$%
, and $\mbox{\bf X}_{\tau _{k}}$ denotes the log efficient price $%
\mbox{\bf
X}_{t}$ at time $\tau _{k}$ for
\begin{equation*}
T-h:=\tau _{0}<\tau _{1}<\cdots <\tau _{n}:=T.
\end{equation*}

Zheng and Li (2011) demonstrated that ${{\mathrm{tr}}\left( {%
\sum_{k=1}^{n}\Delta \mbox{\bf X}_{{k}}\Delta \mbox{\bf X}_{{k}}^{\prime }}%
\right) }/{p}$ is a good estimator for $\int_{T-h}^{T}\gamma _{t}^{2}dt$ and
$\breve{\mbox{\bf S}}_{T-h,T}$ is similar to the sample covariance matrix
with IID samples. Here similarity means that $\breve{\mbox{\bf S}}_{T-h,T}$
is a consistent estimator of population covariance matrix $\pmb\Lambda \pmb%
\Lambda ^{\prime }$ when $p$ is fixed, while the limiting spectral
distribution of $\breve{\mbox{\bf S}}_{T-h,T}$, which will be introduced
later in the paper, is equivalent to that of the sample covariance matrix of
IID samples generated from a distribution with zeros mean and population
covariance $\pmb\Lambda \pmb\Lambda ^{\prime }$, when $p$ goes to $\infty $
together with the sample size $n$.

Clearly, the construction of TVA requires a synchronous record of $p$ assets
at $\left( \tau _{0},\tau _{1},\cdots ,\tau _{n}\right) $. Since data is
always non-synchronous, we need to synchronize them. In this paper, we use
the previous tick method (see Zhang, 2011) to interpolate the prices.
However, the efficient price is latent due to the presence of microstructure
noise. To deal with this problem, we suggest using sparse sampling so that
the impact of microstructure noise can be ignored. Based on a Hausman type
test, A\"{\i}t-Sahalia and Xiu (2016) showed that when data are sampled
every 15 minutes, the observed prices are free of the microstructure noise
problem. In this paper, we will follow this suggestion by sampling the
interpolated data every 15 minutes. Denote $\left( \tau _{0},\tau
_{1},\cdots ,\tau _{n}\right) $ the time stamps at every 15 minutes. So $%
\mbox{\bf Y}_{\tau _{k}}\approx \mbox{\bf X}_{\tau _{k}}$.

Denote the sparsely-sampled log-prices by $\mbox{\bf Y}_{\tau _{0}},%
\mbox{\bf Y}_{\tau _{1}},...,\mbox{\bf Y}_{\tau _{n}}$. The feasible TVA
realized covariance matrix is constructed as
\begin{equation}
\widetilde{\mbox{\bf S}}_{T-h,T}^{\mathrm{TVA}}=\dfrac{{\mathrm{tr}}\left( {%
\sum_{k=1}^{n}\Delta \mbox{\bf Y}_{{k}}\Delta \mbox{\bf Y}_{{k}}^{\prime }}%
\right) }{n}\sum_{k=1}^{n}\dfrac{\Delta \mbox{\bf Y}_{{k}}\Delta \mbox{\bf Y}%
_{{k}}^{\prime }}{|\Delta \mbox{\bf Y}_{k}|^{2}},  \label{S_TVAtilde}
\end{equation}%
with $\Delta \mbox{\bf Y}_{{k}}=\mbox{\bf Y}_{\tau _{k}}-\mbox{\bf Y}_{\tau
_{k-1}}$. Since $\widetilde{\mbox{\bf S}}_{T-h,T}^{\mathrm{TVA}}$ has the
same properties as ${\mbox{\bf S}}_{T-h,T}^{\mathrm{TVA}}$, we treat $%
\widetilde{\mbox{\bf S}}_{T-h,T}^{\mathrm{TVA}}$ the same as ${\mbox{\bf S}}%
_{T-h,T}^{\mathrm{TVA}}$ and only use ${\mbox{\bf S}}_{T-h,T}^{\mathrm{TVA}}$
in the rest of this paper.

It is well-known that the eigenvalues of the sample covariance matrix are
more spread out than those of the population covariance matrix. This
property is applicable not only to the sample covariance matrix but also to $%
\mbox{\bf S}_{T-h,T}^{\mathrm{TVA}}$. In other words, the smallest
eigenvalues of $\mbox{\bf S}_{T-h,T}^{\mathrm{TVA}}$ tend to be biased
downwards, while the largest ones upwards. As a result, there is a need to
regularize the eigenvalues of $\mbox{\bf S}_{T-h,T}^{\mathrm{TVA}}$.

\subsection{Theoretical background for regularizing the eigenvalues of $
\mbox{\bf S}_{T-h,T}^{\mathrm{TVA}}$} \label{Theoretical background for regularizing}

Let us first introduce some concepts in the random matrix theory. Let $p$
denote the number of variables and $n=n(p)$ the sample size. For any $%
p\times p$ symmetric matrix $\mbox{\bM}$, suppose that its eigenvalues are $%
\lambda _{1},\cdots ,\lambda _{p}$, sorted in the non-increasing order. Then
the empirical spectral distribution (ESD) of $\mbox{\bM}$ is defined as
\begin{equation*}
F^{\mbox{\bM}}(x):=\dfrac{1}{p}\sum_{i=1}^{p}\ \mathbb{I}(\lambda _{i}\leq
x),~~~~~\mathrm{for}~~x\in \mathbb{R},
\end{equation*}
where $\mathbb{I}$ denotes the indicator function of a set. The limit of ESD
as $p\rightarrow \infty $, if exists, is referred to the limiting spectral
distribution (LSD hereafter). Let $\mathrm{Supp(G)}$ denotes the support
interval of distribution function $G$. For any distribution $G$, $%
s_{G}(\cdot )$ denotes its Stieltjes transform defined as
\begin{equation*}
s_{G}(z)=\int \dfrac{1}{\lambda -z}dG(\lambda ),~~~~\mathrm{for}~z\in
\mathbb{C}^{+}:=\{z\in \mathbb{C}:\Im (z)>0\},
\end{equation*}%
where $\Im (\cdot )$ denotes the imaginary part of a complex number.

\subsubsection{The limit of loss function}

Suppose the eigen-decomposition of ${\mbox{\bf S}}_{T-h,T}^{\mathrm{TVA}}$
is
\begin{equation}
{\mbox{\bf S}}_{T-h,T}^{\mathrm{TVA}}=\mbox{\bf U}\mbox{\bf V}\mbox{\bf U}%
^{\prime }=\mbox{\bf U}\mathrm{diag}(v_{1},...,v_{p})\mbox{\bf U}^{\prime },
\label{eige-decom of TVA}
\end{equation}%
where $v_{1},...,v_{p}$ are eigenvalues of ${\mbox{\bf S}}_{T-h,T}^{\mathrm{%
TVA}}$ sorted in the non-increasing order, $\mbox{\bf U}=(\mbox{\bf u}%
_{1},...,\mbox{\bf u}_{p})$ are corresponding eigenvectors. Let ${\mathrm{%
diag}}(\mbox{\bf M})$ denote a diagonal matrix with the diagonal elements
being the diagonal elements of $\mbox{\bf M}$ if $\mbox{\bf M}$ is a matrix
or being $\mbox{\bf M}$ if $\mbox{\bf M}$ is a vector.

To regularize the eigenvalues of $\mbox{\bf S}_{T-h,T}^{\mathrm{TVA}}$,
following Ledoit and Wolf (2014), we restrict our attention to a class of
rotation-equivalent estimators which is defined below. This strategy allows
us to use a nonlinear shrinkage method to regularize the eigenvalues.
However, different from Ledoit and Wolf (2014), we do not assume returns are
IID. Instead we assume that $\mbox{\bf X}_{t}$ $\in \mathcal{C}$.

\begin{definition}
\label{class_S}(Class of Estimators $\mathcal{S}$). We consider a generic
positive definite estimator for $\mbox{\boldmath $\Sigma$}_{T-h,T}$ of the
type $\widehat{\mbox{\boldmath $\Sigma$}}_{T-h,T}^{\ast }:={\mbox{\bf U}}{%
\mathrm{diag}}(g_{n}(v_{1}),\cdots ,{g}_{n}(v_{p})){\mbox{\bf U}}^{\prime }$%
, with $v_{1}\geq \cdots \geq v_{p}$ being the eigenvalues of $%
\mbox{\bf
		S}_{T-h,T}^{\mathrm{TVA}}$, ${\mbox{\bf U}}=(\mbox{\bf u}_{1},...,%
\mbox{\bf u}_{p})$ being corresponding eigenvectors. Here $g_{n}$ is a real
univariate function and can depend on ${\mbox{\bf S}}_{T-h,T}^{\mathrm{TVA}}$%
. We assume that there exists a nonrandom real univariate function ${g}(x)$,
defined on $\mathrm{Supp(F)}$ and continuously differentiable, {such that $%
g_{n}(x)\overset{a.s.}{\longrightarrow }g(x)$}, for all $x\in \mathrm{Supp(F)%
}$, where $F$ denotes the LSD of $\mbox{\bf S}_{T-h,T}^{\mathrm{TVA}}$.
\end{definition}

Here, $g_{n}(x)$ is called the \textit{shrinkage function} because what it
does is to shrink the eigenvalues of $\mbox{\bf S}_{T-h, T}^{\mathrm{TVA}}$
by reducing the dispersion around the mean, pushing up the small ones and
pulling down the large ones. The high dimensional asymptotic properties of $%
\mbox{\bf S}_{T-h, T}^{\mathrm{TVA}}$ are fully characterized by its
limiting shrinkage function $g(x)$. As noted in Stein (1975) and Ledoit and
Wolf (2014), the estimators in this class are rotation equivalent, a
property that is desired when the user does not have any prior preference
about the orientation of the eigenvectors.

Since we consider the case that $p$ goes to $\infty $ together with the
sample size, finding the optimal estimator of $\mbox{\boldmath
$\Sigma$}_{T-h,T}$ within class $\mathcal{S}$ for portfolio selection is
equivalent to finding the optimal shrinkage function $g(x)$ that minimizes
the limit of the loss function $\mathcal{L}\left( \widehat{\mbox{\boldmath
$\Sigma$}}_{T-h,T}^{\ast },\mbox{\boldmath $\Sigma$}_{T-h,T}\right) $ for $%
\widehat{\mbox{\boldmath
$\Sigma$}}_{T-h,T}^{\ast }\in \mathcal{S}$. We have the following theorem to
show the limit of $\mathcal{L}\left( \widehat{\mbox{\boldmath
$\Sigma$}}_{T-h,T}^{\ast },\mbox{\boldmath $\Sigma$}_{T-h,T}\right) $.

\begin{thm}
\label{thm:limit_loss} Suppose that $\mbox{\bf X}_t$ is a $p$ -dimensional
diffusion process in class $\mathcal{C}$ for some drift process $%
\mbox{\boldmath $\mu$}_t$, covolatility process $\pmb\Theta_t=\gamma_t\pmb
\Lambda$ and $p$-dimensional Brownian motion $\mbox{\bf B}_t$, which
satisfies the following assumptions:

\begin{description}
\item[(A.i)] $\mbox{\boldmath $\mu$}_t=0$ for $t\in[T-h, T]$, and $\gamma_t$
is independent of  $\mbox{\bf B}_t$.

\item[(A.ii)] There exists $C_0<\infty$ such that for all $p$, $%
|\gamma_t|\in (1/C_0,C_0)$ for all $t\in[T-h, T]$ almost surely;

\item[(A.iii)] All eigenvalues of $\breve{\mbox{\boldmath $\Sigma$}}=\pmb
\Lambda\pmb\Lambda^{\prime }$ are bounded uniformly from 0 and infinity;

\item[(A.iv)] $\lim_{p\rightarrow \infty }{\mathrm{tr}}\left( %
\mbox{\boldmath $\Sigma$}_{T-h, T}\right) /p=\lim_{p\rightarrow \infty
}\int_{T-h}^T\gamma _{t}^{2}dt:=\theta >0$ almost surely;

\item[(A.v)] Almost surely, as $p\rightarrow \infty $, the ESD of $%
\mbox{\boldmath
 		$\Sigma$}_{T-h, T}$ converges to a probability distribution $H$ on a
finite support;

\item[(A.vi)] The observation time points $\tau _{k}$'s are independent of
the Brownian motion $\mbox{\bf B}_{t}$ and there exists a constant $C_{1}>0$
such that $\max_{1\leq k \leq n}n(\tau _{k }-\tau _{k-1})\leq C_{1}$.
\end{description}

If $p/n\rightarrow y\in (0,\infty )$, then the ESD of $\mbox{\bf S}_{T-h,T}^{%
\mathrm{TVA}}$ converges almost surely to a nonrandom probability
distribution $F$. If Equation (\ref{eige-decom of TVA}) is satisfied, then
\begin{equation*}
p\times\mathcal{L}\left( \widehat{\mbox{\boldmath $\Sigma$}}_{T-h,T}^{\ast },%
\mbox{\boldmath $\Sigma$}_{T-h,T}\right) \overset{a.s.}{\rightarrow } \int
\frac{x}{|1-y-yx\times \breve{s}_{F}(x)|^{2}g(x)}dF(x)/\left( \int \frac{%
dF(x)}{g(x)}\right) ^{2},
\end{equation*}
where $\widehat{\mbox{\boldmath $\Sigma$}}_{T-h,T}^{\ast }:={\mbox{\bf U}}{%
\mathrm{diag}}(g_{n}(v_{1}),\cdots ,{g}_{n}(v_{p})){\mbox{\bf
U}}^{\prime }$ is in class $\mathcal{S}$ by regularizing $\mbox{\bf S}%
_{T-h,T}^{\mathrm{TVA}}$, $g(x)$ is the limiting shrinkage function of $%
\widehat{\mbox{\boldmath
$\Sigma$}}_{T-h,T}^{\ast }$. In addition, for all $x\in (0,\infty )$, $%
\breve{s}_{F}(x)$ is defined as $\lim_{z\in \mathbb{C}^{+}\rightarrow x
}s_{F}(z)$, and $s_{F}(z)$ is the Stieltjes transform of the limiting
spectral distribution of $\mbox{\bf S}_{T-h,T}^{\mathrm{TVA}}$.
\end{thm}

\begin{rem}
Theorem \ref{thm:limit_loss} extends the result in Proposition 3.1 of Ledoit
and Wolf (2014) from the IID case to Class $\mathcal{C}$ and from the sample
covariance to the TVA realized covariance.
\end{rem}

\begin{rem}
Without loss of generality, if we assume that all the eigenvalues of $%
\widehat{\mbox{\boldmath $\Sigma$}}_{T-h,T}^{\ast }$ and $\mbox{\boldmath
$\Sigma$}_{T-h,T}$ are bounded, ${\bm{1}^{\prime }\left( \widehat{%
\mbox{\boldmath $\Sigma$} }_{T-h,T}^{\ast }\right) ^{-1}\mbox{\boldmath
$\Sigma$}_{T-h,T}\left(  \widehat{\mbox{\boldmath $\Sigma$}}_{T-h,T}^{\ast
}\right) ^{-1}\bm{1}}=O_p(p)$  and $ {\ \bm{1}^{\prime }\left( \widehat{%
\mbox{\boldmath $\Sigma$}} _{T-h,T}^{\ast }\right) ^{-1}\bm{1} }=O_p(p)$, so
that $\mathcal{L}\left( \widehat{\mbox{\boldmath $\Sigma$}}_{T-h,T}^{\ast },%
\mbox{\boldmath $\Sigma$}_{T-h,T}\right) =O_p(\frac{1}{p})$. This is why we
investigate the limiting behavior of $p\times\mathcal{L}\left( \widehat{%
\mbox{\boldmath $\Sigma$}}_{T-h,T}^{\ast },\mbox{\boldmath $\Sigma$}%
_{T-h,T}\right) $ in Theorem \ref{thm:limit_loss}. %
\end{rem}

\begin{lem}
\label{lem:func_d}Under the assumptions of Theorem \ref{thm:limit_loss}, a
generic positive-definite estimator $\widehat{\mbox{\boldmath $\Sigma$}}%
_{T-h,T}^*$ within class $\mathcal{S}$ minimizes the almost sure limit of
the loss function $\mathcal{L}\left( \widehat{ \mbox{\boldmath $\Sigma$}}%
_{T-h,T}^*,\mbox{\boldmath $\Sigma$}_{T-h,T}\right) $ if and only if its
limiting shrinkage function $g$ satisfies
\begin{equation}
g(x)=\dfrac{x}{|1-y-yx\times \breve{m}_{F}(x)|^{2}},~~~~~\forall ~x\in
\mathrm{Supp(F).}  \label{limit eqn of eigenvalues}
\end{equation}
\end{lem}

Lemma \ref{lem:func_d} is a direct conclusion from Theorem \ref%
{thm:limit_loss} and Proposition 4.1 of Ledoit and Wolf (2014).
Unfortunately, the above minimization problem does not yield a closed-form
solution for $g(x)$ because of $\breve{m}_{F}(x)$ is unknown. In addition,
finding $\breve{m}_{F}(x)$ and then $g(x)$ is numerically difficult in
practice. Finding a good algorithm for estimating $\breve{m}_{F}(x)$ is of
great interest as it was done in Ledoit and Wolf (2014) that used a
commercial package. However, in this paper we propose to find an alternative
interpretation of $g(x)$, which offers an easier way to approximate $g(x)$.

\subsubsection{Alternative interpretation of $g(x)$}

Motivated from Ledoit and P\`{e}ch\`{e} (2011), we can show that $g(x)$ in (%
\ref{limit eqn of eigenvalues}) is equivalent to the asymptotic quantity
corresponding to the oracle nonlinear shrinkage estimator derived from the
following Frobenius norm of the difference between $\mbox{\bf U}\widetilde{%
\mbox{\bf V}}\mbox{\bf U}^{\prime }$ and $\mbox{\boldmath
	$\Sigma$}_{T-h,T}$, i.e.,
\begin{equation*}
\min_{\widetilde{\mbox{\bf V}}~\mathrm{\ diagonal}}\Vert \mbox{\bf U}%
\widetilde{\mbox{\bf V}}\mbox{\bf U}^{\prime }-\mbox{\boldmath $\Sigma$}%
_{T-h,T}\Vert _{F},
\end{equation*}%
where the Frobenius norm is defined as $\Vert \mbox{\bf M}\Vert _{F}=\sqrt{%
\mathrm{tr(\mbox{\bf M}\mbox{\bf M}^{\prime })}}$ for any real matrix $%
\mbox{\bf M}$.

Elementary matrix algebra shows that the solution is
\begin{equation}
\widetilde{\mbox{\bf V}} =\mathrm{diag}(\tilde{v}_{1},\cdots ,\tilde{v}%
_{p}),~~~~\mathrm{where}~\tilde{v}_{i}=\mbox{\bf u} _{i}^{\prime }%
\mbox{\boldmath $\Sigma$}_{T-h,T}\mbox{\bf u}_{i},~i=1,\cdots ,p.
\label{sol of min fro}
\end{equation}%
To characterize the asymptotic behavior of $\tilde{v}_{i},i=1,\cdots ,p$,
following the idea of Ledoit and P\`ech\`e (2011), we define the following
non-decreasing function
\begin{equation}  \label{eqn:ESD_d}
\Psi _{p}(x)=\dfrac{1}{p}\sum_{i=1}^{p}\tilde{v}_{i}~\mathbb{I}(v _{i}\leq
x)=\dfrac{1}{p}\sum_{i=1}^{p}\mbox{\bf u}_{i}^{\prime }%
\mbox{\boldmath
$\Sigma$}_{T-h,T}\mbox{\bf u}_{i}\cdot \mathbb{I}(v _{i}\leq x).
\end{equation}

\begin{thm}
\label{thm:LSD_d}Assume that assumptions (A.i)-(A.vi) in Theorem \ref%
{thm:limit_loss} hold true and let $\Psi _{p}$ be defined as in %
\eqref{eqn:ESD_d}. If $p/n\rightarrow y\in (0,\infty )$, then there exists a
nonrandom function $\Psi $ defined over $\mathbb{R}$ such that $\Psi _{p}(x)$
converges almost surely to $\Psi (x)$ for all $x\in \mathbb{R}\backslash
\{0\}$. If in addition $y\neq 1$, then $\Psi $ can be expressed as
\begin{equation}
\forall~ x\in \mathbb{R},~~~\Psi (x)=\int_{-\infty }^{x}\delta (v )dF(v ),
\label{eqn_psi}
\end{equation}
where $F$ is the LSD of $\mbox{\bf S}_{T-h, T}^{\mathrm{TVA}}$, and if $v >0$%
,
\begin{equation*}
\delta (v )=\dfrac{v}{|1-y-y v\times \breve{m}_{F}(v)|^{2} }.
\end{equation*}
\end{thm}

\begin{rem}
Theorem \ref{thm:LSD_d} extends the result in Theorem 4 of Ledoit and
P\`ech\`e (2011) from the IID case to Class $\mathcal{C}$.
\end{rem}

Theorem \ref{thm:LSD_d} implies that the asymptotic quantity that
corresponds to $\tilde{v}_{i}={\mbox{\bf u}}_{i}^{\prime }%
\mbox{\boldmath
	$\Sigma$}_{T-h,T}{\mbox{\bf u}}_{i}$ is $\delta (v)$ provided that $v$
corresponds to $v_{i}$. An interesting finding is that the results of Lemma %
\ref{lem:func_d} and Theorem \ref{thm:LSD_d} are consistent with each other,
even though they are motivated from two different perspectives. Given that
it is much easier to work on the minimization problem in (\ref{sol of min
fro}), we recommend to regularize the eigenvalues of $\mbox{\bf S}_{T-h,T}^{%
\mathrm{TVA}}$ by using (\ref{sol of min fro}), which is to find a good
estimator for each $\tilde{v}_{i}={\mbox{\bf u}}_{i}^{\prime }%
\mbox{\boldmath
		$\Sigma$}_{T-h,T}{\mbox{\bf u}}_{i}$ with $i=1,...,p$ .


\subsection{Regularized estimators of eigenvalues of $\mbox{\bf S}_{T-h,T}^{%
\mathrm{TVA}}$} \label{Regularized estimators of eigenvalues}

\label{Estimating the eigenvalues}

Note that $\tilde{v}_{i}=\mbox{\bf u}_{i}^{\prime }\mbox{\boldmath $\Sigma$}%
_{T-h,T}\mbox{\bf u}_{i}$ is actually the integrated volatility of process $%
\mbox{\bf u}_{i}^{\prime }\mbox{\bf X}_{t}$ over $[T-h, T]$ for $%
i=1,2,\cdots ,p$. A natural estimator of each $\tilde{v}_{i}$ is the
realized volatility $\sum_{k=1}^{n}(\mbox{\bf u}_{i}^{\prime }\Delta %
\mbox{\bf X}_{k})^{2}$. Unfortunately, this is not a good idea. To see the
problem, note that
\begin{equation*}
\widehat{\mbox{\boldmath $\Sigma$}}_{T-h,T}^{\ast \ast }=\mbox{\bf U}{%
\mathrm{diag}}\left( \sum_{k=1}^{n}(\mbox{\bf u}_{1}^{\prime }\Delta %
\mbox{\bf X}_{k})^{2},...,\sum_{k=1}^{n}(\mbox{\bf u}_{p}^{\prime }\Delta %
\mbox{\bf X}_{k})^{2}\right) \mbox{\bf U}^{\prime }.
\end{equation*}%
Let us consider the simplest case where $\gamma _{t}=1$, $\bm\Lambda =%
\mbox{\bf I}_{p}$ with $\mbox{\bf I}_{p}$ be a $p$-dimensional identity
matrix, and $\tau _{k}-\tau _{k-1}=\frac{h}{n}$ for $k=1,...,n$. We can
write $\Delta \mbox{\bf X}_{k}=\left( \frac{h}{n}\right) ^{1/2}\mbox{\bf Z}%
_{k}$ with $\mbox{\bf Z}_{k}$'s are IID $p$-dimensional standard normals
such that $\dfrac{\Delta \mbox{\bf X}_{{k}}\Delta \mbox{\bf X}_{{k}}^{\prime
}}{|\Delta \mbox{\bf X}_{k}|^{2}}=\frac{\mbox{\bf Z}_{k}\mbox{\bf Z}%
_{k}^{\prime }}{|\mbox{\bf Z}_{k}|^{2}}$. Since $|\mbox{\bf Z}_{k}|^{2}\sim
p $ as $p\rightarrow \infty $, we have
\begin{align*}
\mbox{\bf S}_{T-h,T}^{\mathrm{TVA}}& =\dfrac{{\mathrm{tr}}\left( {\
\sum_{k=1}^{n}\Delta \mbox{\bf X}_{k}\Delta \mbox{\bf X}_{k}^{\prime }}%
\right) }{p}\dfrac{p}{n}\sum_{k=1}^{n}\dfrac{\Delta \mbox{\bf X}_{{k}}\Delta %
\mbox{\bf X}_{{k}}^{\prime }}{|\Delta \mbox{\bf X}_{k}|^{2}} \\
& \sim \dfrac{{\mathrm{tr}}\left( {\ \sum_{k=1}^{n}\Delta \mbox{\bf X}%
_{k}\Delta \mbox{\bf X}_{k}^{\prime }}\right) }{p}\dfrac{1}{n}\sum_{k=1}^{n}%
\mbox{\bf Z}_{k}\mbox{\bf Z}_{k}^{\prime }, \\
\sum_{k=1}^{n}{\Delta \mbox{\bf X}_{{k}}\Delta \mbox{\bf X}_{{k}}^{\prime }}%
& =\frac{h}{n}\sum_{k=1}^{n}\mbox{\bf Z}_{k}\mbox{\bf Z}_{k}^{\prime }.
\end{align*}%
By denoting $\Delta \mbox{\bf X}=(\Delta \mbox{\bf X}_{1},...,\Delta %
\mbox{\bf X}_{n})^{\prime }$, we have
\begin{align*}
\widehat{\mbox{\boldmath $\Sigma$}}_{T-h,T}^{\ast \ast }& =\mbox{\bf U}%
\mathrm{diag}\left( \sum_{k=1}^{n}(\mbox{\bf u}_{1}^{\prime }\Delta %
\mbox{\bf X}_{k})^{2},...,\sum_{k=1}^{n}(\mbox{\bf u}_{p}^{\prime }\Delta %
\mbox{\bf X}_{k})^{2}\right) \mbox{\bf U}^{\prime } \\
& =\mbox{\bf U}\mathrm{diag}\left( \mbox{\bf u}_{1}^{\prime }\Delta %
\mbox{\bf X}\Delta \mbox{\bf X}^{\prime }\mbox{\bf u}_{1},...,\mbox{\bf u}%
_{p}^{\prime }\Delta \mbox{\bf X}\Delta \mbox{\bf X}^{\prime }\mbox{\bf u}%
_{p}\right) \mbox{\bf U}^{\prime } \\
& =\mbox{\bf U}\mathrm{diag}\left( \mbox{\bf U}^{\prime }\Delta \mbox{\bf X}%
\Delta \mbox{\bf X}^{\prime }\mbox{\bf U}\right) \mbox{\bf U}^{\prime } \\
& \sim \Delta \mbox{\bf X}\Delta \mbox{\bf X}^{\prime },
\end{align*}%
which is actually the sample covariance matrix of IID samples generated from
$N(\bm0,h\mbox{\bf I}_{p})$. Hence, its eigenvalues are also more spread out
than that of $h\mbox{\bf I}_{p}$, a well-known result in the literature.

To solve this problem, we use the idea from Abadir et al. (2014) and Lam
(2016) by splitting the sample into two parts. We use the estimated
eigenvectors from a fraction of the data to transform the data into
approximately orthogonal series.\footnote{%
Strictly speaking, the asymptotic justification of the method requires the
IID assumption as shown in Lam  (2016). While the IID assumption does not
hold for Class $\mathcal{C}$, we examine the effectiveness of this method
using real data later.} We then use the independence of two sample
covariance matrices to regularize the eigenvalues of one of them. Therefore,
instead of based $\mbox{\bf U}$ on $\Delta \mbox{\bf X}_{k}=\mbox{\bf X}%
_{\tau _{k}}-\mbox{\bf
X}_{\tau _{k-1}}~(k=1,...,n)$ for $T-h:=\tau _{0}<...<\tau _{n}:=T$, we base
$\mbox{\bf U}^{\ast }$ on $\Delta \mbox{\bf X}_{r}^{\ast }=\mbox{\bf X}%
_{\tau _{r}^{\ast }}-\mbox{\bf X}_{\tau _{r-1}^{\ast }}(r=1,...,m)$ for
\begin{equation*}
0:=\tau _{0}^{\ast }<\tau _{1}^{\ast }<...<\tau _{m}^{\ast }<T-h,
\end{equation*}%
where $\mbox{\bf U}^{\ast }=(\mbox{\bf u}_{1}^{\ast },...,\mbox{\bf u}%
_{p}^{\ast })$ are the eigenvectors of $\mbox{\bf S}_{0,T-h}^{\mathrm{TVA}}$
corresponding to the eigenvalues with the non-increasing order, and the TVA
realized covariance matrix
\begin{equation*}
\mbox{\bf S}_{0,T-h}^{\mathrm{TVA}}=\dfrac{{\mathrm{tr}}\left\{ {\
\sum_{r=1}^{m}\Delta \mbox{\bf X}_{r}^{\ast }\left( \Delta \mbox{\bf X}%
_{r}^{\ast }\right) ^{\prime }}\right\} }{p}\cdot \breve{\mbox{\bf S}}%
_{0,T-h},~~\mathrm{with}~~\breve{\mbox{\bf S}}_{0,T-h}=\dfrac{p}{m}%
\sum_{r=1}^{m}\dfrac{\Delta \mbox{\bf X}_{r}^{\ast }\left( \Delta
\mbox{\bf
X}_{r}^{\ast }\right) ^{\prime }}{|\Delta \mbox{\bf X}_{r}^{\ast }|^{2}}.
\end{equation*}%
In addition, since the eigenvectors of $\mbox{\boldmath $\Sigma$}_{t}$ is
assumed to be time invariant, we also consider the following optimization
problem
\begin{equation*}
\min_{{\mbox{\bf V}}^{\ast }~\mathrm{\ diagonal}}\Vert \mbox{\bf U}^{\ast }{%
\mbox{\bf V}}^{\ast }\left( \mbox{\bf U}^{\ast }\right) ^{\prime }-%
\mbox{\boldmath $\Sigma$}_{T-h,T}\Vert _{F},
\end{equation*}%
and estimate each diagonal element of the oracle minimizer ${\mbox{\bf V}}%
^{\ast }={\mathrm{diag}}({v}_{1}^{\ast },...,{v}_{p}^{\ast })$ with ${v}%
_{i}^{\ast }=\left( \mbox{\bf
u}_{i}^{\ast }\right) ^{\prime }\mbox{\boldmath $\Sigma$}_{T-h,T}\mbox{\bf u}%
_{i}^{\ast }$ based on the data over the time period $[T-h,T]$. To get an
accurate estimator for each ${v}_{i}^{\ast }$ with $i\in \{1,...,p\}$, we
propose to use all the tick-by-tick high frequency data and take into
account with the microstructure noises.

Let us first consider the case that the data are synchronous and equally
recorded at time points $\{T-h:=t_{0}^{\ast }<t_{1}^{\ast }<\cdots
<t_{N}^{\ast }:=T\}$, where the time interval ${\Delta }=t_{j}^{\ast
}-t_{j-1}^{\ast }\rightarrow 0$ for all $j=1,...,N$ as $N\rightarrow \infty $
and $h$ fixed. Notice that here $\{t_{j}^{\ast }:j=0,...,N\}$ may be quite
different from $\{\tau _{k}:k=0,...,n\}$ and ${\Delta }$ can be one second
or a few seconds, and should be much smaller than $\tau _{k}-\tau _{k-1}$
which is 15 minutes.

We assume each observation is contaminated by microstructure noise such that
$\mbox{\bf Y}_{t}=(Y_{1t}, ..., Y_{pt})^{\prime }$ (observed) contains the
true log-price $\mbox{\bf X}_{t}$ (latent) and the microstructure noise ${\bm%
\epsilon }_{t}=(\epsilon _{1t},\dots ,\epsilon _{pt})^{\prime }$ in an
additive form
\begin{equation}
\mbox{\bf Y}_{t}=\mbox{\bf X}_{t}+\bm\epsilon _{t},\text{ for }t\in \lbrack
T-h,T],  \label{observed and latent}
\end{equation}%
where the $p$-dimensional noise ${\bm\epsilon }_{t}$ is assumed to satisfy

\begin{assumption}
The $p$-dimensional noise ${\bm\epsilon }_{t}=(\epsilon _{1t},\dots
,\epsilon _{pt})^{\prime }$ at different time points $t=t_{0}^{\ast
},t_{1}^{\ast },\cdots ,t_{N}^{\ast }$ are IID random vectors with mean $\bm%
0 $ (a $p$-dimensional vector with all elements being 0), positive definite
covariance matrix $\mbox{\bf A}_{0}$ and finite fourth moment. In addition, $%
{\bm\epsilon }_{t}$ and $\mbox{\bf X}_{t}$ are mutually independent.
\end{assumption}

This assumption has commonly been used in the literature; see, for example, A%
\"{\i}t-Sahalia et al. (2010), Zhang (2011), Liu and Tang (2014). To
estimate $(\mbox{\bf u}_{i}^{\ast })^{\prime }\mbox{\boldmath
$\Sigma$}_{T-h,T}\mbox{\bf u}_{i}^{\ast }$, we apply the quasi-maximum
likelihood (QML) approach developed in Xiu (2010). Based on (\ref{eqref})
and (\ref{observed and latent}), we have
\begin{align}
\tilde{Y}_{it}& =(\mbox{\bf u}_{i}^{\ast })^{\prime }\mbox{\bf Y}_{t}=(%
\mbox{\bf u}_{i}^{\ast })^{\prime }\mbox{\bf X}_{t}+(\mbox{\bf u}_{i}^{\ast
})^{\prime }{\bm\epsilon }_{t}=\tilde{X}_{it}+\tilde{\epsilon}_{it}  \notag
\\
d\tilde{X}_{it}& =(\mbox{\bf u}_{i}^{\ast })^{\prime }d\mbox{\bf X}_{t}=(%
\mbox{\bf
u}_{i}^{\ast })^{\prime }\bm{\mu}_{t}dt+(\mbox{\bf u}_{i}^{\ast })^{\prime }%
\pmb\Theta _{t}d\mbox{\bf B}_{t}=\tilde{{\mu }}_{it}dt+\tilde{\sigma}_{it}d%
\tilde{B}_{it}  \label{model for QML}
\end{align}%
by letting
\begin{align*}
\tilde{X}_{it}& =(\mbox{\bf u}_{i}^{\ast })^{\prime }\mbox{\bf X}%
_{t},~~~~~~~~~~~~~~~~\tilde{\epsilon}_{it}=(\mbox{\bf u}_{i}^{\ast
})^{\prime }{\bm\epsilon }_{t},~~~~~~~~~~~~~~\tilde{{\mu }}_{it}=(%
\mbox{\bf
u}_{i}^{\ast })^{\prime }\bm{\mu}_{t}, \\
\tilde{\sigma}_{it}d\tilde{B}_{it}& =(\mbox{\bf u}_{i}^{\ast })^{\prime }\pmb%
\Theta _{t}d\mbox{\bf B}_{t},~~\tilde{\sigma}_{it}^{2}=(\mbox{\bf u}%
_{i}^{\ast })^{\prime }\pmb\Theta _{t}((\mbox{\bf u}_{i}^{\ast })^{\prime }%
\pmb\Theta _{t})^{\prime }=(\mbox{\bf u}_{i}^{\ast })^{\prime }\pmb\Theta
_{t}\pmb\Theta _{t}^{\prime }\mbox{\bf u}_{i}^{\ast }=(\mbox{\bf u}%
_{i}^{\ast })^{\prime }\mbox{\boldmath $\Sigma$}_{t}\mbox{\bf u}_{i}^{\ast },
\end{align*}%
such that $v_{i}^{\ast }=\int_{T-h}^{T}\tilde{\sigma}_{it}^{2}dt$.

Ignoring the impact of $\tilde{\mu}_{it}dt$ by considering $\tilde{\mu}%
_{it}=0$, we follow the idea in Xiu (2010) to give two misspecified
assumptions for each $i\in \{1,...,p\}$. First, the spot volatility is
assumed to be time invariant:  $\tilde{\sigma}_{it}^{2}=(\mbox{\bf u}%
_{i}^{\ast })^{\prime }\mbox{\boldmath $\Sigma$}_{t}\mbox{\bf u}_{i}^{\ast }=%
\tilde{\sigma}_{i}^{2}$. Second, the noise $\tilde{\epsilon}_{it}$ is
assumed to be normally distributed with mean 0 and variance $\tilde{a}%
_{i}^{2}$. Then the quasi-log likelihood function for $\tilde{Y}%
_{i,t_{j}^{\ast }}-\tilde{Y}_{i,t_{j-1}^{\ast }}$ is
\begin{equation}
l(\tilde{\sigma}_{i}^{2},\tilde{a}_{i}^{2})=-\frac{1}{2}{\log }~\mathrm{det}(%
\mathbf{\Omega }^{\ast })-\frac{Np}{2}{\log }(2\pi )-\frac{1}{2}\left(
\widetilde{\mbox{\bf Y}}_{i}^{\ast }\right) ^{\prime }(%
\mbox{\boldmath
$\Omega$}^{\ast })^{-1}\left( \widetilde{\mbox{\bf Y}}_{i}^{\ast }\right)
\label{qmlf-xiu}
\end{equation}%
where $\mathbf{\Omega }^{\ast }$ is a tridiagonal matrix with the diagonal
elements being $\tilde{\sigma}_{i}^{2}{\Delta }+2\tilde{a}_{i}^{2}$ and the
tridiagonal elements being $-\tilde{a}_{i}^{2}$, $\tilde{\mbox{\bf Y}}%
_{i}^{\ast }=\left( \tilde{Y}_{i,t_{1}^{\ast }}-\tilde{Y}_{i,t_{0}^{\ast
}},...,\tilde{Y}_{i,t_{N}^{\ast }}-\tilde{Y}_{i,t_{N-1}^{\ast }}\right)
^{\prime }$. The QML estimator of $\left( \int_{T-h}^{T}\tilde{\sigma}%
_{it}^{2}dt,(\mbox{\bf u}_{i}^{\ast })^{\prime }\mbox{\bf A}_{0}\mbox{\bf u}%
_{i}^{\ast }\right) $ is the value of $(\tilde{\sigma}_{i}^{2},\tilde{a}%
_{i}^{2})$ which maximizes $l(\tilde{\sigma}_{i}^{2},\tilde{a}_{i}^{2})$. We
denote the estimator of ${v}_{i}^{\ast }=\int_{T-h}^{T}\tilde{\sigma}%
_{it}^{2}dt$ by $\hat{v}_{i}^{\ast }$, which is positive. Xiu (2010) proved
that $\hat{v}_{i}^{\ast }$ is consistent and asymptotically efficient for $%
\int_{T-h}^{T}\tilde{\sigma}_{it}^{2}dt$.

\begin{rem}
As discussed in Xiu (2010), if $(t_{j}^{\ast }-t_{j-1}^{\ast })$s for $%
j=1,...,N$ are random and IID, we can add another misspecified assumption
that they are equal. We then apply the above approach to get $\hat{v}%
_{i}^{\ast }$ which is also a consistent estimator of $(\mbox{\bf u}%
_{i}^{\ast })^{\prime }\mbox{\boldmath $\Sigma$}_{T-h,T}\mbox{\bf u}%
_{i}^{\ast }$. Since the tick-by-tick data over the time period $[T-h,T]$ is
typically non-synchronous, we propose to first synchronize data by the
refresh time scheme of Barndorff-Nielsen et al. (2011) and then apply the
QML procedure to obtain $\hat{v}_{i}^{\ast }~(i=1,\cdots ,p)$. The first
refresh time $t_{0}^{\ast }$ during a trading day is the first time when all
assets have been traded at least once since $T-h$. The second refresh time $%
t_{1}^{\ast }$ is the first time when all assets have been traded at least
once since the first refresh point in time $t_{0}^{\ast }$. Repeating this
sequence yields in total $N+1$ refresh times, $t_{0}^{\ast },t_{1}^{\ast
},...,t_{N}^{\ast }$, and corresponding $N+1$ sets of synchronized refresh
prices $\mathbf{Y}_{t_{0}^{\ast }},\mathbf{Y}_{t_{1}^{\ast }},...,\mathbf{Y}%
_{t_{N}^{\ast }}$ with each $Y_{i,t_{j}^{\ast }}~~(i=1,...,p;j=0,1,...,N)$
being the log-price of the $i$th asset nearest to and previous to $%
t_{j}^{\ast }$. Barndorff-Nielsen et al. (2011) showed that if the trading
time of $p$ assets arrive as independent standard Poisson processes with
common intensity $\lambda $ such that the mean of trading frequency of each
asset over $[T-h,T]$ is $\lambda h$, then the synchronized data obtained by
the refresh time scheme is $\lambda h/\log p$. Based on this observation, if
each of 100 (or1,000) assets have around 20,000 observations within a
trading day, then the number of synchronized observations is around 4,342
(or 2,895). While this sampling strategy loses around 78.3\% or 85.5\% of
observations, it keeps much more data than the sparsely sampling technique
at every 15 minutes, where the size is only 26 within a trading day.
\end{rem}

Therefore, our shrinkage QML estimators for $\mbox{\boldmath $\Sigma$}%
_{T-h,T}$ and $\mbox{\boldmath $\Sigma$}_{T-h,T}^{-1}$ are, respectively,
\begin{equation}
\widehat{\mbox{\boldmath $\Sigma$}}_{T-h,T}=\mbox{\bf U}^{\ast }\mathrm{\
diag}(\hat{v}_{1}^{\ast },...,\hat{v}_{p}^{\ast })\left( \mbox{\bf U}^{\ast
}\right) ^{\prime },~~~\widehat{\mbox{\boldmath $\Sigma$}_{T-h,T}^{-1}}=%
\mbox{\bf U}^{\ast }\mathrm{diag}\left\{ (\hat{v}_{1}^{\ast })^{-1},...,(%
\hat{v}_{p}^{\ast })^{-1}\right\} \left( \mbox{\bf
U}^{\ast }\right) ^{\prime },  \label{SQMLE}
\end{equation}%
and our estimated optimal weight $\hat{\mbox{\bf w}}_{T}$ is obtained by
replacing $\mbox{\boldmath $\Sigma$}_{T-h,T}^{-1}$ in (\ref%
{min_risk_optimal_weightB}) with $\widehat{\mbox{\boldmath $\Sigma$}%
_{T-h,T}^{-1}}$,
\begin{equation}
\hat{\mbox{\bf w}}_{T}=\frac{\widehat{\mbox{\boldmath $\Sigma$}_{T-h,T}^{-1}}%
{\bm1}}{{\bm1}^{\prime }\widehat{\mbox{\boldmath $\Sigma$}_{T-h,T}^{-1}}{\bm1%
}}.  \label{QML Weight}
\end{equation}%
Notice that like $\mbox{\bf U}$, $\mbox{\bf U}^{\ast }$ cannot be obtained
directly from observations. We therefore approximate $\mbox{\bf U}^{\ast }$
by the eigenvectors of
\begin{equation*}
\widetilde{\mbox{\bf S}}_{0,T-h}^{\mathrm{TVA}}=\dfrac{{\mathrm{tr}}\left\{ {%
\ \sum_{r=1}^{m}\Delta \mbox{\bf Y}_{{r}}^{\ast }\left( \Delta \mbox{\bf Y}_{%
{r}}^{\ast }\right) ^{\prime }}\right\} }{m}\sum_{r=1}^{m}\dfrac{\Delta %
\mbox{\bf Y}_{{r}}^{\ast }\left( \Delta \mbox{\bf Y}_{{r}}^{\ast }\right)
^{\prime }}{|\Delta \mbox{\bf Y}_{{r}}^{\ast }|^{2}},
\end{equation*}%
where $\Delta \mbox{\bf Y}_{r}^{\ast }=\mbox{\bf Y}_{\tau _{r}^{\ast }}-%
\mbox{\bf Y}_{\tau _{r-1}^{\ast }}~(r=1,...,m)$, and $\mbox{\bf Y}_{\tau
_{r}^{\ast }}$'s are the log-prices obtained by synchronizing all the trading
prices of $p$ assets during $[0,T-h)$ via the previous tick method.

\section{Empirical Studies}

\label{Empirical Study} In this section, we demonstrate the performance of
our proposed method using real data. Three portfolio sizes are considered ($%
p=30,40$ and $50$) based on stocks traded in the U.S. markets. These
portfolios are 30 Dow Jones Industrial Average (30 DJIA) constituent stocks,
30 DJIA stocks and 10 stocks with the largest market caps (ranked on March
30, 2012) from S\&P 500 other than 30 DJIA stocks, 30 DJIA stocks and 20
stocks with the largest market caps from S\&P 500 other than 30 DJIA stocks.
We download daily data starting from March 19, 2012 and ending on December
31, 2013 (450 trading days) from the Center for Research in Security Prices
(CRSP) and 200 days intra-day data staring on March 19, 2013 and ending on
December 31, 2013 from the TAQ database. The daily data are used to
implement some existing methods in the literature for the purpose of
comparison. For the high frequency data, the same data cleaning procedure as
in Barndorff-Nielsen et al. (2011) is applied to pre-process the data by 1)
deleting entries that have 0 or negative prices, 2) deleting entries with
negative values in the column of \textquotedblleft Correlation
Indicator\textquotedblright , 3) deleting entries with a letter code in the
column of \textquotedblleft COND\textquotedblright , except for
\textquotedblleft E\textquotedblright\ or \textquotedblleft
F\textquotedblright , 4) deleting entries outside the period 9:30 a.m. to 4
p.m., and 5) using the median price if there are multiple entries at the
same time.

\subsection{Summary of the proposed method}

Given that, in the empirical applications, the basic unit is daily, we can
summarize the proposed method as follows. Suppose we want to construct a
portfolio strategy at the end of the $J$th day (which is denoted $T$ in
previous sections) based on a pool of $p$ assets with a holding period of $%
\breve{J}$ days. We use the ICV in the most recent $J-J_{1}$ days (which is
denoted $[T-h,T]$ in previous sections) multiplied by $\frac{\breve{J}}{%
J-J_{1}}$ to approximate the expected ICV during the holding period.\newline
\textbf{Step 1:} Split data of $J$ days into two parts. The first part
contains data of first $J_{1}$ days, recorded as the $1$st, ..., $J_{1}$th
days. The rest of data of $J-J_{1}$ days belong to the second part.\newline
\textbf{Step 2:} Synchronize data in the $l$th day for each $l\in
\{1,...,J_{1}\}$ using the previous tick method at the 15-minute interval.
Denote the log-price at the 15-minute frequency by $\mbox{\bf Y}_{0},%
\mbox{\bf Y}_{1},...,\mbox{\bf Y}_{m}$.\newline
\textbf{Step 3:} Synchronize the data in $l$th day for each $l\in
\{J_{1}+1,...,J\}$ using the refresh time scheme to obtain synchronous data
and denote the log-price by $\mbox{\bf Y}_{\cdot 0}^{l\ast },\mbox{\bf Y}%
_{\cdot 1}^{l\ast },...,\mbox{\bf Y}_{\cdot n_{l}}^{l\ast }$ for each $l\in
\{J_{1}+1,...,J\}$.\newline
\textbf{Step 4: } Obtain the eigenvectors of $\dfrac{{\mathrm{tr}}\left( {\
\sum_{k=1}^{m}\Delta \mbox{\bf Y}_{{k}}\Delta \mbox{\bf Y}_{{k}}^{\prime }}%
\right) }{m}\sum_{k=1}^{m}\dfrac{\Delta \mbox{\bf Y}_{{k}}\Delta \mbox{\bf Y}%
_{{k}}^{\prime }}{|\Delta \mbox{\bf Y}_{k}|^{2}}$ (the corresponding
eigenvalues are sorted in the non-increasing order), and put them together
as a $p\times p$ matrix which is denoted by $\mbox{\bf U}^{\ast }$. Here $%
\Delta \mbox{\bf Y}_{{k}}=\mbox{\bf Y}_{k}-\mbox{\bf Y}_{k-1}$. \newline
\textbf{Step 5: } Obtain $\widetilde{\mbox{\bf Y}}_{\cdot j}^{l\ast }=\left( %
\mbox{\bf U}^{\ast }\right) ^{\prime }\mbox{\bf Y}_{\cdot j}^{l\ast }$ for $%
l=J_{1}+1,...,J,~j=1,...,n_{l}$. Estimate the integrated volatility of the $%
i $th element of $\left( \mbox{\bf U}^{\ast }\right) ^{\prime }\mbox{\bf X}%
_{t} $ during the $l$th day by QML that maximizes (\ref{qmlf-xiu}) with $%
\widetilde{\mbox{\bf Y}}_{i}^{\ast }$ being replaced by $\widetilde{%
\mbox{\bf Y}}_{i\cdot }^{l\ast }=\left( \tilde{Y}_{i1}^{l\ast },...,\tilde{Y}%
_{i,n_{l}}^{l\ast }\right) $ and with $\tilde{Y}_{ij}^{l\ast }$ being the $i$%
th element of $\widetilde{\mbox{\bf Y}}_{\cdot j}^{l\ast }$. Denote the
estimator by $\hat{v}_{i}^{l\ast }$.\newline
\textbf{Step 6: } The SQML estimator of the ICV in the $l$th day is defined
as $\mbox{\bf U}^{\ast }{\mathrm{diag}}(\hat{v}_{1}^{l\ast },...,\hat{v}%
_{p}^{l\ast })\left( \mbox{\bf U}^{\ast }\right) ^{\prime }$. We then use $%
\frac{\breve{J}}{J-J_{1}}\sum_{l=J_{1}+1}^{J}\mbox{\bf U}^{\ast }{\mathrm{%
diag}}(\hat{v}_{1}^{l\ast },...,\hat{v}_{p}^{l\ast })\left( \mbox{\bf U}%
^{\ast }\right) ^{\prime }$ to approximate the expected ICV during the
holding period, and its inverse to approximate $\widehat{%
\mbox{\boldmath
$\Sigma$}_{T-h,T}^{-1}}$ in (\ref{QML Weight}) to get the estimated optimal
weight.

For the purpose of comparison, we consider two different $\mbox{\bf U}^{\ast
}$s. We denote the two different SQML estimators by SQrM if $\mbox{\bf U}%
^{\ast }$ in Step 4 is obtained from 15-minute intra-day data and SQrD if $%
\mbox{\bf Y}_{0},...,\mbox{\bf Y}_{m}$ are the daily closing log-prices.

\subsection{ The GMV portfolio}

We first consider the GMV portfolio problem (\ref{gmv}) whose theoretical
optimal weight is chosen by (\ref{min_risk_optimal_weight}). Following the
choice of many practitioners, we apply the plug-in method to estimate the
optimal weight and replace $\widetilde{\mbox{\boldmath $\Sigma$}}_{T,T+\tau
}^{-1}$ by its approximation, $\frac{h}{\tau }\widehat{%
\mbox{\boldmath
$\Sigma$}_{T-h,T}^{-1}}$ with different $h$s. We refer to Brandt (2010) for
a review of the impacts of a plug-in method in portfolio choice.

We compare the out-of-sample performance of our proposed method with some
other methods in the literature, including the equal weight (denoted by EW),
the weight estimated by plugging in the optimal linear shrinkage of the
sample covariance matrix (denoted by LS), the weight derived by the
procedure suggested in Fan, Li and Yu (2012) (denoted by TS). After the
weights are determined, the portfolios are constructed accordingly.

LS is obtained by replacing $\widetilde{ \mbox{\boldmath $\Sigma$}}%
_{T,T+\tau }^{-1}$ in (\ref{min_risk_optimal_weight}) with the inverse of
the linear shrinkage estimator
\begin{equation*}
\widehat{\mbox{\bf S}}_{LS}=(1-\kappa )\mbox{\bf S}+\kappa \bar{\lambda}%
\mbox{\bf I}_{p},
\end{equation*}%
where $\mbox{\bf S}=J_{LS}^{-1}\sum_{i=1}^{J_{LS}}(\mbox{\bf Y}_{i}-%
\mbox{\bf Y}_{i-1})(\mbox{\bf Y}_{i}-\mbox{\bf Y}_{i-1})^{\prime }=%
\mbox{\bf
Q}{\mathrm{\ diag}}(\lambda _{1},...,\lambda _{p})\mbox{\bf Q}^{\prime }$ is
the sample covariance matrix of previous $J_{LS}$ daily log-returns, $%
\lambda _{1},...,\lambda _{p}$ are the eigenvalues of $\mbox{\bf S}$, $%
\mbox{\bf Q}$ contains corresponding eigenvectors, $\bar{\lambda}%
=\sum_{i=1}^{p}\lambda _{i}/p$, and $\kappa $ is determined by the
asymptotic optimization results derived in Ledoit and Wolf (2004).

Fan, Li and Yu (2012) considered the following risk optimization problem
under gross-exposure constraints
\begin{equation}
\min \mbox{\bf w}^{\prime }\mbox{\boldmath $\Sigma$}_{T-h,T}\mbox{\bf w}%
~~~~s.t.~~~\Vert \mbox{\bf w}\Vert _{1}\leq c~~~\text{and}~~~\mbox{\bf w}%
^{\prime }\bm{1}=1,  \label{two scale portfolio}
\end{equation}%
where $\mbox{\boldmath $\Sigma$}_{T-h,T}$ was also used to approximate $%
\widetilde{\mbox{\boldmath $\Sigma$}}_{T,T+\tau }$. The pair-wise two scales
covariance (TSCV) estimator of $\mbox{\boldmath $\Sigma$}_{T-h,T}$ was
constructed based on the high frequency data synchronized by the pair-wise
refresh time scheme over previous $J_{TS}$ trading days. Since this
pair-wise estimator may not be positive semi-definite, they projected the
estimator (denoted by $\mbox{\bf M}$ here) by
\begin{equation}
\mbox{\bf M}_{1}=(\mbox{\bf M}+\lambda _{\min }^{-}\mbox{\bf I}%
_{p})/(1+\lambda _{\min }^{-}),  \label{projection of TS}
\end{equation}%
where $\lambda _{\min }^{-}$ is the negative part of the minimum eigenvalue
of the estimator $\mbox{\bf M}$. They then minimize $\mbox{\bf w}^{\prime }%
\mbox{\bf M}_{1}\mbox{\bf w}$ to obtain the optimal weight $\widehat{%
\mbox{\bf w}}$ for a given $c$. In this paper, following the simulation and
the empirical studies in Fan, Li and Yu (2012), we set $c=1.2$.

In practice, one choice that we have to make is the number of days over
which we do the estimation. For our new developed approach, we let $%
J_{1}=50,60,...,250$ when we use daily log-returns, and let $J_{1}=$ 5 (one
week), 6, ..., 21 (one month) days when we use 15-minute intra-day
log-returns in Step 4. Moreover, we choose $J-J_{1}=1,2,....,5$. The optimal
result among all possible combinations is reported. Similarly, we report the
optimal results for LS when $J_{LS}\in \{50,60,...,250\}$ and TS when $%
J_{TS}\in \{1,2,....,10\}$, and denote them by TSo, LSo respectively.

The following three measures are calculated to compare the out-of-sample
performance of all the methods during 174 investment days (we have 200 days
intra-day data in total and we use 26 days intra-day data to get SQrM), from
April 25, 2013 to December 31, 2013: (1) the average of log-returns of the
portfolio multiplied by 252 (denoted by AV); (2) the standard deviation of
log-returns of the portfolio multiplied by $\sqrt{252}$ (denoted by SD); (3)
information ratio calculated by AV/SD (denoted by IR).

In general, a high AV and a high IR with a low SD are expected for a good
portfolio. Since the GMV portfolio is designed to minimize the variance of a
portfolio, the most important performance measure for GMV is SD. Therefore,
we first compare the standard deviations of different methods and then
compare the information ratios and the average returns.

Reported in Table 1 are the AV, SD and IR for all the methods. The number in
the bold face represents the lowest SD. Several conclusions can be made from
Table 1. First and foremost, SQrM outperforms all the other strategies in
terms of SD. SQrM also achieves the highest information ratio when $p=50$.
Second, as expected, the standard deviation of the GMV portfolio decreases,
as $p$ increases from 30 to 50, for most methods. The only exception is the
EW. Third, SQrM performs better than SQrD, indicating that high frequency
data are useful in portfolio choice.


\subsection{Markowitz portfolio with momentum signals (MwM)}

We now consider a `full' Markowitz portfolio without any short-sale constraint.
The Markowitz portfolio minimizes the variance of a portfolio under two
conditions:
\begin{equation*}
\min \mbox{\bf w}^{\prime }\widetilde{\mbox{\boldmath $\Sigma$}}_{T,T+\tau }%
\mbox{\bf w}~~~~\text{subject ~to}~~\mbox{\bf w}^{\prime }\bm{1}=1~~\text{and%
}~~\mbox{\bf w}^{\prime }\mbox{\bf e}=b,
\end{equation*}%
where $b$ is a target expected return chosen by an investor and $\mbox{\bf e}
$ is a signal to denote the vector of expected returns of $p$ assets. The
above problem has the following analytical solution
\begin{equation}
\mbox{\bf w}=c_{1}\widetilde{\mbox{\boldmath $\Sigma$}}_{T,T+\tau }^{-1}%
\bm{1}+c_{2}\widetilde{\mbox{\boldmath $\Sigma$}}_{T,T+\tau }^{-1}%
\mbox{\bf
	e},  \label{Markowtiz With Momentum}
\end{equation}%
where
\begin{equation*}
c_{1}=\frac{C-bB}{AC-B^{2}},~~c_{2}=\frac{bA-B}{AC-B^{2}},~~A=\bm{1}^{\prime
}\widetilde{\mbox{\boldmath $\Sigma$}}_{T,T+\tau }^{-1}\bm{1},~~B=\bm{1}%
^{\prime }\widetilde{\mbox{\boldmath $\Sigma$}}_{T,T+\tau }^{-1}\bm{e},~~C=%
\bm{e}^{\prime }\widetilde{\mbox{\boldmath $\Sigma$}}_{T,T+\tau }^{-1}\bm{e}.
\end{equation*}%
To choose $\mbox{\bf e}$ and $b$, we follow Ledoit and Wolf (2014). In
particular, the $i$th element of $\mbox{\bf e}$ is the momentum factor which
is chosen as the arithmetic average of the previous 250 days returns on the $%
i$th stock. $b$ is the arithmetic average of the momentums of the
top-quintile stocks according to $\mbox{\bf e}$. In Table 2, we report the
annualized AV, SD, and IR of the daily log-returns for all methods, namely,
SQrM, SQrD, LS, the equal weight constructed on top-quintile stocks
according to their momentums (denoted by EW-TQ), and the method with $%
\widetilde{\mbox{\boldmath
		$\Sigma$}}_{T,T+\tau }^{-1}$ in (\ref{Markowtiz With Momentum}) being
replaced by the inverse of the sample covariance matrix of previous $J_{SP}$
days daily log-returns (denoted by SP when $J_{SP}=250$). Similar to the GMV
portfolio, we choose optimal $J$ and $J_{1}$ for SQrM and SQrD. For the
Markowitz portfolio, a more relevant criterion for the comparison is IR. In
this paper, we first compare the IRs and then the SDs.

In Table 2, the number in bold face represents the highest IR while the
number with a `*' represents the lowest SD. It can be seen that SQrM and
SQrD perform better than other methods in terms of IR except the EW when $%
p=30, 40$. However, the SDs of SQrM and SQrD are much lower than that of EW
and also lower than that of the other methods.

\subsection{Robustness of sample period}

To check the robustness of our strategy, we split the entire 174 investment
days into two subperiods, one from April 25, 2013 to August 27, 2013 and the
other from August 28, 2013 to December 31, 2013. The results of the GMV
portfolio are reported in Tables 3 and 4. It can be seen that SQrM and SQrD
continue to outperform other methods in terms of SD in all cases. Empirical
results of the Markowitz portfolio with the momentum signal are reported in
Tables 5 and 6. Again SQrM and SQrD continue to outperform other methods in
almost all cases in terms of IR and SD.

We also perform a moving-window analysis to check the robustness of our
empirical results. Staring from April 25, 2013, we calculate the standard
deviation of daily log-returns of each method over 42 trading days and
repeat this exercise by moving one trading day at each pass. To compare our
method with other methods, we use figures to show the results of TSo, LSo,
SQrD and SQrM for the GMV portfolio and SPo, LSo, SQrD and SQrM for the MwM
portfolio, where the optimal numbers of days chosen for each method is to
minimize or maximize the mean of SDs or IRs of 133 different investment
periods (each investment period is 42 days) for the GMV and MwM portfolio,
respectively. Figures 1, 2, 3 plot the results when $p=$ 30, 40 and 50.

We find that SQrM performs better and better as the portfolio size
increases. In general, it has the lowest SD and the highest IR for both the
GMV portfolio and the MwM portfolios. This result indicates that high
frequency data are useful in portfolio choice, especially for controlling
the risk.

\subsection{Robustness of time span}

From a statistical perspective, a longer span of historical data contains
more information about the dynamic of an asset price so that it may be
reasonable to believe that methods based on a longer span of data should
perform better than those based on less data. However, the model
specification is more likely to be wrong over a longer span. Hence there is
a trade off between the estimation error and the specification error. In
this subsection we examine this trade off empirically in the context of the
LS portfolio and the SQrD portfolio. In particular, the LS portfolio and the
SQrD portfolio are constructed based on different historical data sets for
the GMV portfolio. For LS, we set $J_{LS}=50,60,...,250$. For SQrD, we fix $%
J-J_{1}=1$ and set $J_{1}=50,60,...,250$.

Figure 4 plots the risk of daily log-returns for the two GMV portfolios\ as
a function of $J_{LS}$ or $J_{1}$ when $p=30$ and $p=50$. Some interesting
findings emerge. First, the risk of log-returns of a portfolio does not
necessarily decrease when a longer span of historical data is used. Second,
SQrD performs better than LS in almost all cases and is more stable across
different time spans. This is especially true when $p=50$. Once again, there
is an advantage for using our estimator for portfolio selection. Third, when
$p=30$, the risk of the SQrD portfolio decreases when $J_{1}$ increases
initially. This is because more data are used in estimation, reducing the
estimation error. However, the risk increases when $J_{1}>110$. This is
because the construction of SQrD relies on the assumption that $\mbox{\bf X}%
_{t}\in \mathcal{C}$. As $J_{1}$ increases, the time span becomes longer,
and hence the assumption that $\mbox{\bf X}_{t}\in \mathcal{C}$ is more
likely to be invalid. This can also explain why SQrM performs better than
SQrD.

\section{Conclusions}

This paper has developed a new estimator for the ICV and its inverse from
high frequency data when the portfolio size $p$ and the sample size of data $%
n$ satisfies $p/n\rightarrow y>0$ as $n$ goes to $\infty $. The use of high
frequency data drastically increases the sample size and hence reduces the
estimation error. To further prevent the estimation error from accumulating
with $p$, a new regularization method is applied to the eigenvalues of an
initial estimator of the ICV. Our proposed estimator of the ICV is always
positive definite and its inverse is the estimator of the inverse of the
ICV. It minimizes the limit of the out-of-sample variance of portfolio
returns within the class of rotation-equivalent estimators. It works when
the number of underlying assets is larger than the number of time series
observations in each asset and when the asset price follows a general
stochastic process.

The asymptotic optimality for our proposed method is justified under the
assumption that $p/n\rightarrow y>0$ as $n$ goes to $\infty $. The
usefulness of our estimator is examined in real data. The method is used to
construct the optimal weight in the global minimum variance and the
Markowitz portfolio with momentum signal based on the DJIA 30 and another 20
stocks chosen from S\&P500. The performance of our proposed method is
compared with that of some existing methods in the literature. The empirical
results show that our method performs favorably out-of-sample.


\section{Appendix}

In the appendix we first prove Theorem \ref{thm:LSD_d} as the proof of
Theorem \ref{thm:limit_loss} relies on Theorem \ref{thm:LSD_d}.

\begin{proof}[Proof of  Theorem \ref{thm:LSD_d}]
	By assumption (A.i), we can write
	\[
	\De \bX_k  \ = \  \int_{\tau_{k-1}}^{\tau_{k}} \ga_t \bLa d\bW_t
	\  \eqd \  \(  \int_{\tau_{k-1}}^{\tau_{k}} \ga_t^2 dt \)^{1/2} \breve{\bSig}^{1/2} \bz_k,
	\]
	where `$\eqd$' stands for `equal in distribution', $\breve{\bSig}=\bm{\Lambda}\bm{\Lambda}'$ and $\bz_k=(Z_{1k}, \cdots, Z_{pk})'$ consists of independent standard normals.
	Then
	\[
	\bS_{T-h, T}^{\TVA} = \dfrac{{\tr}\left(\widehat{\bSig}_{T-h, T}^{\RCV}\right)}{p} \cdot \dfrac p {n} \sum_{k=1}^n \breve{\bSig}^{1/2} \(\dfrac{\bz_k \bz_{k}' }{\bz'_k \breve{\bSig} \bz_k}\) \breve{\bSig}^{1/2},
	\]
	where ${\widehat{\bSig}_{T-h, T}^{\RCV}=\sum_{k =1}^{n}\Delta \mbox{\bf X}_{k}\Delta \mbox{\bf X}	_{k}^{\prime }}$, $\Delta \mbox{\bf X}	_{k}= \mbox{\bf X}	_{\tau _{k}}-\mbox{\bf X}	 _{\tau _{k-1}}$. Denote
	\[
	\bS_{T-h, T}^{IID}:=\sum_{k =1}^{n} \dfrac 1{n} \bSig_{T-h, T} ^{1/2} \bz_k\bz_{k}' \bSig_{T-h, T} ^{1/2}
	= \int_{T-h}^T\ga_t^2 dt \cdot \( \dfrac 1{n} \sum_{k =1}^{n} \breve{\bSig}^{1/2} \bz_k\bz_{k}' \breve{\bSig}^{1/2} \).
	\]
	
	From Theorem 2 of  Ledoit and P\`ech\`e (2011), we know that $p^{-1}{\tr}\left\{\(\bS_{T-h, T}^{IID}-z\bI\)^{-1} \bSig_{T-h, T} \right\}$ converges to
	\[
	s_\Psi(z)=\int\dfrac{r}{r \{1-y-yz \times s_F(z) \}-z} \ dH(r),
	\]
	almost surely, where $H$ is the LSD of matrices $\bSig_{T-h, T} $ and $F$ is the LSD of matrices $\bS_{T-h, T}^{IID}$ or $\bS_{T-h, T}^{\TVA}$, since they share the same LSD by Theorem 2 of Zheng and Li (2011).

 On the other hand, we have that $s_\Psi(z)$ is the Stieltjes transform of the bounded function $\Psi(x)$ defined in \eqref{eqn_psi} by Theorem 4 of Ledoit and P\`ech\`e (2011) and the Stieltjes transform of function $\Psi_p(x)$ is
	\[
	s_{\Psi_p}(z)=\dfrac 1p {\tr} \left\{ \(\bS_{T-h, T}^{\TVA}-z\bI\)^{-1} \bSig_{T-h, T} \right\}.
	\]
Therefore  we only need to show that
	\[
	\dfrac 1p {\tr} \left\{ \(\bS_{T-h, T}^{IID}-z\bI\)^{-1} \bSig_{T-h, T}  \right\} - \dfrac 1p {\tr} \left\{\(\bS_{T-h, T}^{\TVA}-z\bI\)^{-1} \bSig_{T-h, T} \right\} \stackrel {a.s.}{\to} 0.
	\]
	To prove this, it suffices to show the following two facts:
	\begin{eqnarray}\label{F1}
	\max_{1\le k\le n} \left|\dfrac 1p \bz_k' \breve{\bSig} \bz_k -1 \right| \stackrel {a.s.}{\to} 0,
	\end{eqnarray}
	and
	\begin{eqnarray}\label{F2}
	\dfrac 1p {\tr}\left(\widehat{\bSig}_{T-h, T}^{\RCV}\right) -\int_{T-h}^T \ga_t^2 dt \stackrel {a.s.}{\to} 0.
	\end{eqnarray}
	
	To prove \eqref{F1}, by assumption (A.iii), all the eigenvalues of $\breve{\bSig}$ are bounded, so that ${\tr}(\breve{\bSig}^r)=O(p)$ for all $1\le r<\infty$.
	From Lemma 2.7 of Bai and Silverstein (1998), we have
	\begin{eqnarray*}
		\E \( \max_{1\le k\le n} \left| p^{-1} \bz_k' \breve{\bSig} \bz_k -1 \right|^6 \)
		&\le& \dfrac{Cn}{p^6} \(\left\{\E|Z_{jk}|^4 \tr\left(\breve{\bSig}^2\right) \right\}^3+ \E|Z_{jk}|^{12} \tr\left(\breve{\bSig}^6\right) \) \\
		&=& O(n^{-2}),
	\end{eqnarray*}
	where the last step comes from the fact that the higher order moments of $Z_{jk}$'s are finite since they are normally distributed. Thus, \eqref{F1} follows by the Borel-Cantelli lemma.
	
	We now prove \eqref{F2}.
	\begin{eqnarray*}
		\left| p^{-1}{\tr}\left(\widehat{\bSig}_{T-h, T}^{\RCV}\right) -\int_{T-h}^T \ga_t^2 dt \right|
		&=& \left| p^{-1}\sum_{k=1}^n \int_{\tau_{k-1}}^{\tau_{k}} \ga_t^2dt \cdot \bz_k' \breve{\bSig}\bz_k- \int_{T-h}^T \ga_t^2 dt \right|  \\
		&=& \left| \sum_{k=1}^n  \int_{\tau_{k-1}}^{\tau_{k}} \ga_t^2dt \cdot \(p^{-1}\bz_k' \breve{\bSig} \bz_k -1 \) \right|  \\
		&\le& \max_{1\le k\le n} \left| p^{-1}\bz_k' \breve{\bSig}\bz_k-1 \right| \cdot \int_{T-h}^T \ga_t^2 dt \\
		&\stackrel {a.s.}{\to}& 0.
	\end{eqnarray*}
	by assumption (A.iv) and the result in Equation \eqref{F1}.
	
	Then,
	\begin{eqnarray*}
		&&\dfrac 1p {\tr}\left\{ \(\bS_{T-h, T} ^{IID}-z\bI\)^{-1}\bSig_{T-h, T} \right\} - \dfrac 1p {\tr}\left\{ \(\bS_{T-h, T}^{\TVA}-z\bI\)^{-1}\bSig_{T-h, T} \right\}\\
		&=& \dfrac 1p {\tr}\left\{\(\bS_{T-h, T} ^{IID}-z\bI\)^{-1} \(\bS_{T-h, T}^{\TVA}-\bS_{T-h, T}^{IID}\) \(\bS_{T-h, T}^{\TVA}-z\bI\)^{-1}  \bSig_{T-h, T} \right\}\\
		&=& {\small \dfrac 1p {\tr}\left\{\(\bS_{T-h, T}^{IID}-z\bI\)^{-1} \dfrac{{\tr}\left(\widehat{\bSig}_{T-h, T}^{\RCV}\right)}{p(n)} \cdot  \sum_{k=1}^n \(\dfrac{1}{p^{-1}\bz_k' \breve{\bSig}\bz_k }-1\) \breve{\bSig}^{1/2} \bz_k \bz_{\ell}' \breve{\bSig}^{1/2} \(\bS_{T-h, T}^{\TVA}-z\bI\)^{-1}\bSig_{T-h, T}  \right\}}\\
		&&  + \dfrac 1{p(n)} {\tr}\left( \(\bS_{T-h, T}^{IID}-z\bI\)^{-1} \left\{p^{-1}{\tr}\left(\widehat{\bSig}_{T-h, T}^{\RCV}\right)-\int_{T-h}^T\ga_t^2dt\right\}\right.\\
			&&\left.\times \sum_{k=1}^n \breve{\bSig}^{1/2} \bz_k \bz_{k}' \breve{\bSig}^{1/2} \(\bS_{T-h, T}^{\TVA}-z\bI\)^{-1}\bSig_{T-h, T}   \right)\\
		&:=& I_1+I_2.
	\end{eqnarray*}
	From assumptions (A.ii)-(A.v), and the facts that
	$\| (\bS_{T-h, T}^{IID}-z\bI)^{-1} \| \le 1/\Im(z)$,  $\| (\bS_{T-h, T}^{\TVA}-z\bI)^{-1} \| \le 1/\Im(z)$ with  $\left\|\cdot\right\|$ denoting the $L_2$ norm of a matrix, \eqref{F1} and \eqref{F2}, we have that both $|I_1|$ and $|I_2|$ converge to 0, almost surely.
	Therefore, the proof of Theorem \ref{thm:LSD_d} is completed.
\end{proof}

\begin{proof}[Proof of  Theorem \ref{thm:limit_loss}]
	The convergence of ESD of $\bS_{T-h, T}^{\TVA}$ is shown in Theorem 2 of Zheng and Li (2011).
	Note that $\left\|\left(\widehat{\bSig}_{ T-h, T} ^*\right)^{-1} \right\|\le C$ for some fixed number $C$ when $p$ large enough  by assumption (A.v) and the fact that $\left(\widehat{\bSig}_{ T-h, T} ^*\right)^{-1} $ belongs to class $\cS$.  Thus, from Lemma 2.7 of Bai and Silverstein (1998) and Borel-Cantelli lemma, we have
	\[
	\dfrac 1p \1' \left(\widehat{\bSig}_{ T-h, T} ^*\right)^{-1}  \1 -\dfrac 1p {\tr}\left\{ \left(\widehat{\bSig}_{ T-h, T} ^* \right)^{-1}\right\} \stackrel {a.s.}{\to} 0.
	\]
	Moreover, we have
\begin{eqnarray*}
		\dfrac 1p {\tr}\left\{ \left(\widehat{\bSig}_{ T-h, T} ^*\right)^{-1} \right\}&=& \dfrac 1p \sum_{i=1}^p \dfrac{1}{g_n(v_{i})}
		= \int \dfrac{1}{g_n(x)} dF^{\bS_{T-h, T}^{\TVA}} (x)  \\
		&\stackrel {a.s.}{\to}& \int\dfrac{1}{g(x)} dF(x).
	\end{eqnarray*}
	Therefore,
	\begin{eqnarray}\label{conv_1}
	\dfrac 1p \1'  \left(\widehat{\bSig}_{ T-h, T} ^*\right)^{-1}  \1  \stackrel {a.s.}{\to}   \int\dfrac{1}{g(x)} dF(x).
	\end{eqnarray}
	
	Similarly, we can show that
	\[
	\dfrac 1p \1' \left(\widehat{\bSig}_{ T-h, T} ^*\right)^{-1}  \bSig_{T-h, T}  \left(\widehat{\bSig}_{ T-h, T} ^*\right)^{-1}  \1 - \dfrac 1p {\tr}\left\{\left(\widehat{\bSig}_{ T-h, T} ^*\right)^{-1} \bSig_{T-h, T}  \left(\widehat{\bSig}_{ T-h, T} ^*\right)^{-1} \right\} \stackrel {a.s.}{\to} 0.
	\]
	Using Theorem \ref{thm:LSD_d}, we have
	\begin{eqnarray*}
		\dfrac 1p {\tr}\left\{ \left(\widehat{\bSig}_{ T-h, T} ^*\right)^{-1} \bSig_{T-h, T}  \left(\widehat{\bSig}_{ T-h, T} ^*\right)^{-1} \right\}
		&=& \dfrac 1p {\tr} (\bU' \bSig_{T-h, T} \bU\bV^{-2})
		=\dfrac 1p \sum_{i=1}^p \dfrac{\bu_{i}' \bSig_{T-h, T} \bu_{i} }{g_n(v_{i})^2} \\
		&\stackrel {a.s.}{\to}& \int\dfrac{x}{|1-y-yx\times \breve{m}_F(x)|^2  g(x)^2} dF(x).
	\end{eqnarray*}
	Thus,
	\begin{eqnarray}\label{conv_2}
	\dfrac 1p \1' \left(\widehat{\bSig}_{ T-h, T} ^*\right)^{-1}  \bSig_{T-h, T}  \left(\widehat{\bSig}_{ T-h, T} ^*\right)^{-1}  \1 \stackrel {a.s.}{\to}
	\int\dfrac{x}{|1-y-yx\times \breve{m}_F(x)|^2 g(x)^2} dF(x).
	\end{eqnarray}
	Combining \eqref{conv_1} and \eqref{conv_2}, we obtain that
	\[
	p \cdot \dfrac{\1' \left(\widehat{\bSig}_{ T-h, T} ^*\right)^{-1} \bSig_{0, T-h} \left(\widehat{\bSig}_{ T-h, T} ^*\right)^{-1} \1}{\left(\1' \left(\widehat{\bSig}_{ T-h, T} ^*\right)^{-1}\1\right)^2}
	\stackrel {a.s.}{\to} \dfrac{\ \int\dfrac{x}{|1-y-yx\times \breve{m}_F(x)|^2 g(x)^2} dF(x)}{\(\ \int\dfrac{dF(x)}{g(x)}\)^2}.
	\]
\end{proof}

\begin{table}[tbp]
\caption{The out-of-sample performance of different daily rebalanced
strategies for the GMV portfolio between April 25, 2013 and December 31,
2013.}
\label{GMV Whole}\centering
\begin{tabular}{|c||c|c|c|c|c|c|c|}
\hline
\multicolumn{8}{|c|}{{\small {Period: 04/25/2013---12/31/2013}}} \\ \hline
$p=30$ & EW & TS & TSo & LS & LSo & SQrD & SQrM \\ \hline
AV & 20.13 & 13.22 & 13.22 & 15.31 & 12.96 & 10.59 & 15.62 \\ \hline
SD & 10.17 & 9.65 & 9.65 & 9.80 & 9.52 & 9.34 & \textbf{9.17} \\ \hline
IR & 1.98 & 1.37 & 1.37 & 1.56 & 1.36 & 1.80 & 1.70 \\ \hline\hline
$p=40$ & EW & TS & TSo & LS & LSo & SQrD & SQrM \\ \hline
AV & 21.00 & 16.51 & 17.80 & 16.00 & 11.84 & 19.06 & 18.09 \\ \hline
SD & 10.43 & 9.66 & 9.62 & 9.85 & 9.29 & 9.29 & \textbf{9.10} \\ \hline
IR & 2.01 & 1.71 & 1.85 & 1.62 & 1.27 & 2.05 & 1.99 \\ \hline\hline
$p=50$ & EW & TS & TSo & LS & LSo & SQrD & SQrM \\ \hline
AV & 21.00 & 20.15 & 20.15 & 13.28 & 10.25 & 17.74 & 20.52 \\ \hline
SD & 10.36 & 9.40 & 9.40 & 9.47 & 9.18 & 9.26 & \textbf{8.68} \\ \hline
IR & 2.03 & 2.14 & 2.14 & 1.40 & 1.12 & 1.91 & 2.36 \\ \hline
\end{tabular}%
\par
\begin{flushleft}
{\footnotesize {Note: AV, SD, IR denote the average, standard deviation, and
information ratio of 174 daily log-returns, respectively. AV and SD are
annualized and in percent. The smallest number in the row labeled by SD is
reported in bold face. TSo corresponds to the case where $%
\mbox{\boldmath
$\Sigma$}_{T-h,T}$ is estimated by the two-scale covariance matrix obtained
based on historical intra-day data (10 days when $p=30, 50$; 8 days when $%
p=40$). LSo corresponds to the case where $\mbox{\boldmath $\Sigma$}_{T-h,T}$
is estimated by the linear shrinkage of the sample covariance matrix of
daily log-returns (110, 90 and 90 days when $p=30, 40$ and $50$
respectively). The optimal number of days is chosen by minimizing SD of 174
log-returns of each portfolio.} }
\end{flushleft}
\end{table}

\begin{table}[tbp]
\caption{The out-of-sample performance of different daily rebalanced
strategies for Markowitz portfolio with momentum signal between April 25,
2013 and December 31, 2013.}\centering
\begin{tabular}{|c||c|c|c|c|c|c|c|}
\hline
\multicolumn{8}{|c|}{{\small {Period: 04/25/2013---12/31/2013}}} \\ \hline
$p=30$ & EW-TQ & SP & SPo & LS & LSo & SQrD & SQrM \\ \hline
AV & 31.74 & 0.02 & 4.18 & 13.02 & 13.02 & 20.02 & 15.91 \\ \hline
SD & 13.27 & 12.10 & 12.06 & 11.59 & 11 56 & 11.37 & $11.11^* $ \\ \hline
IR & \textbf{2.39} & 0.00 & 0.35 & 1.12 & 1.12 & {1.76} & 1.43 \\
\hline\hline
$p=40$ & EW & SP & SPo & LS & LSo & SQrD & SQrM \\ \hline
AV & 36.49 & 6.86 & 8.94 & 18.98 & 18.98 & 24.67 & 20.21 \\ \hline
SD & 13.55 & 10.90 & 10.99 & 11.29 & 10.66 & 11.25 & $10.07^* $ \\ \hline
IR & \textbf{2.69} & 0.63 & 0.81 & 1.68 & 1.68 & 2.19 & {2.01} \\
\hline\hline
$p=50$ & EW-TQ & SP & SPo & LS & LSo & SQrD & SQrM \\ \hline
AV & {28.64} & 7.53 & 11.33 & 15.81 & 15.81 & {22.65} & 23.35 \\ \hline
SD & 13.16 & 10.42 & 10.80 & 10.60 & 10.60 & 10.51 & $9.83^*$ \\ \hline
IR & 2.18 & 0.72 & 1.05 & 1.49 & 2.03 & 2.15 & \textbf{2.38} \\ \hline
\end{tabular}%
\par
\begin{flushleft}
{\footnotesize {Note: AV, SD, IR denote the average, standard deviation, and
information ratio of 174 daily log-returns respectively. AV, SD are
annualized and in percent. The smallest number in the row labeled by SD is
reported in bold face. SPo corresponds to the case where $%
\mbox{\boldmath
$\Sigma$}_{T-h,T}$ is estimated by the sample covariance matrix of daily
log-returns (190, 230 and 130 days when $p=30, 40$ and $50$, respectively).
LSo corresponds to the case where $\mbox{\boldmath
$\Sigma$}_{T-h,T}$ is estimated by the linear shrinkage of the sample
covariance matrix of daily log-returns (250 days when $p=30, 40$ and $50$).
The optimal number of days is chosen by maximizing IR of 174 log-returns of
each portfolio. } }
\end{flushleft}
\end{table}

\begin{table}[tbp]
\caption{The out-of-sample performance of different daily rebalanced
strategies for the GMV portfolio between April 25, 2013 and August 27, 2013.}%
\centering
\begin{tabular}{|c||c|c|c|c|c|c|c|}
\hline
\multicolumn{8}{|c|}{{\small {Period: 04/25/2013---08/27/2013}}} \\ \hline
$p=30$ & EW & TS & TSo & LS & LSo & SQrD & SQrM \\ \hline
AV & 5.23 & 2.23 & 2.23 & 2.63 & 9.90 & 2.93 & 1.87 \\ \hline
SD & 10.93 & 10.05 & 10.05 & 10.64 & 9.85 & 9.97 & \textbf{\ 9.81} \\ \hline
IR & 0.48 & 0.22 & 0.22 & 0.25 & 1.00 & 0.29 & 0.19 \\ \hline\hline
$p=40$ & EW & TS & TSo & LS & LSo & SQrD & SQrM \\ \hline
AV & 5.04 & 3.64 & 4.35 & 0.48 & 4.49 & 0.43 & 0.42 \\ \hline
SD & 11.09 & 10.26 & 10.00 & 10.51 & 9.45 & 9.78 & \textbf{9.73 } \\ \hline
IR & 0.45 & 0.36 & 0.43 & 0.05 & 0.48 & 0.04 & 0.04 \\ \hline\hline
$p=50$ & EW & TS & TSo & LS & LSo & SQrD & SQrM \\ \hline
AV & 6.53 & 9.35 & 9.35 & -1.78 & 8.29 & 1.70 & 5.70 \\ \hline
SD & 11.06 & 10.12 & 10.12 & 10.18 & \textbf{9.18} & 9.89 & {\ 9.30} \\
\hline
IR & 0.59 & 0.92 & 0.92 & -0.17 & 0.89 & 0.17 & 0.61 \\ \hline
\end{tabular}%
\par
\begin{flushleft}
{\footnotesize {Note: AV, SD, IR denote the average, standard deviation, and
information ratio of 87 daily log-returns, respectively. AV, SD are
annualized and in percent. The smallest number in the row labeled by SD is
reported in bold face. TSo corresponds to the case where $%
\mbox{\boldmath
$\Sigma$}_{T-h,T}$ is estimated by the two-scale covariance matrix obtained
based on historical intra-day data (10 days when $p=30, 50$; 8 days when $%
p=40$). LSo corresponds to the case where $\mbox{\boldmath $\Sigma$}_{T-h,T}$
is estimated by the linear shrinkage of the sample covariance matrix of
daily log-returns (110, 90 and 90 days when $p=30, 40$ and $50$
respectively). } }
\end{flushleft}
\end{table}


\begin{table}[tbp]
\caption{The out-of-sample performance of different daily rebalanced
strategies for the GMV portfolio between August 28, 2013 to December 31,
2013.}\centering
\begin{tabular}{|c||c|c|c|c|c|c|c|}
\hline
\multicolumn{8}{|c|}{{\small {Period: 08/28/2013---12/31/2013}}} \\ \hline
$p=30$ & EW & TS & TSo & LS & LSo & SQrD & SQrM \\ \hline
AV & 35.04 & 24.22 & 24.22 & 27.99 & 16.02 & 30.70 & 29.36 \\ \hline
SD & 9.31 & 9.24 & 9.24 & 8.87 & 9.17 & 8.64 & \textbf{8.45} \\ \hline
IR & 3.76 & 2.62 & 2.62 & 3.15 & 1.75 & 3.55 & 3.48 \\ \hline\hline
$p=40$ & EW & TS & TSo & LS & LSo & SQrD & SQrM \\ \hline
AV & 36.97 & 29.37 & 31.25 & 31.52 & 19.19 & 20.15 & 25.94 \\ \hline
SD & 9.69 & 9.01 & 9.19 & 9.09 & 9.16 & 8.54 & \textbf{8.32} \\ \hline
IR & 3.82 & 3.26 & 3.40 & 3.47 & 2.09 & 4.411 & 4.30 \\ \hline\hline
$p=50$ & EW & TS & TSo & LS & LSo & SQrD & SQrM \\ \hline
AV & 35.47 & 30.95 & 30.95 & 28.33 & 12.20 & 33.78 & 35.34 \\ \hline
SD & 9.60 & 8.64 & 8.64 & 8.67 & 9.05 & 8.52 & \textbf{\ 7.95} \\ \hline
IR & 3.70 & 3.58 & 3.58 & 3.27 & 1.35 & 3.96 & 4.44 \\ \hline
\end{tabular}%
\par
\begin{flushleft}
{\footnotesize {Note: AV, SD, IR denote the average, standard deviation, and
information ratio of 87 daily log-returns, respectively. AV, SD are
annualized and in percent. The smallest number in the row labeled by SD is
reported in bold face. TSo corresponds to the case where $%
\mbox{\boldmath
$\Sigma$}_{T-h,T}$ is estimated by the two-scale covariance matrix obtained
based on historical intra-day data (10 days when $p=30, 50$ and 8 days when $%
p=40$). LSo corresponds to the case where $\mbox{\boldmath $\Sigma$}_{T-h,T}$
is estimated by the linear shrinkage of the sample covariance matrix of
daily log-returns (110, 90 and 90 days when $p=30, 40$ and $50$
respectively). } }
\end{flushleft}
\end{table}

\begin{table}[tbp]
\caption{The out-of-sample performance of different daily rebalanced
strategies for the Markowitz portfolio with momentum signal between April
25, 2013 to August 27, 2013.}\centering
\begin{tabular}{|c||c|c|c|c|c|c|c|}
\hline
\multicolumn{8}{|c|}{{\small {Period: 04/25/2013---08/27/2013}}} \\ \hline
$p=30$ & EW-TQ & SP & SPo & LS & LSo & SQrD & SQrM \\ \hline
AV & 8.70 & -10.05 & -1.43 & -2.99 & -2.99 & 1.73 & 0.13 \\ \hline
SD & 14.51 & 13.02 & 13.03 & 12.54 & 12.54 & $11.94^*$ & 12.06 \\ \hline
IR & \textbf{\ 0.60 } & -0.81 & -0.11 & -0.24 & -0.24 & 0.01 & 0.04 \\
\hline\hline
$p=40$ & EW-TQ & SP & SPo & LS & LSo & SQrD & SQrM \\ \hline
AV & 9.21 & -11.39 & -9.39 & -4.04 & -4.04 & 4.14 & -2.69 \\ \hline
SD & 14.19 & 10.91 & 10.94 & 11.45 & 10.99 & 11.49 & $10.42^* $ \\ \hline
IR & \textbf{0.65} & -1.04 & -0.86 & -0.35 & -0.35 & 0.36 & -0.26 \\
\hline\hline
$p=50$ & EW-TQ & SP & SPo & LS & LSo & SQrD & SQrM \\ \hline
AV & 3.50 & -8.15 & 1.71 & -6.25 & -6.25 & {6.98} & 4.16 \\ \hline
SD & 14.28 & 11.01 & 10.48 & 10.92 & 10.92 & 10.84 & $10.47^*$ \\ \hline
IR & 0.25 & -0.74 & 0.16 & -0.57 & -0.57 & \textbf{0.64} & 0.40 \\ \hline
\end{tabular}%
\par
\begin{flushleft}
{\footnotesize {Note: AV, SD, IR denote the average, standard deviation, and
information ratio of 87 daily log-returns respectively. AV, SD are
annualized and in percent. The smallest number in the row labeled by SD is
reported in bold face. SPo corresponds to the case where $%
\mbox{\boldmath
$\Sigma$}_{T-h,T}$ is estimated by the sample covariance matrix of daily
log-returns (190, 230 and 130 days when $p=30, 40$ and $50$ respectively).
LSo corresponds to the case where $\mbox{\boldmath
$\Sigma$}_{T-h,T}$ is estimated by the linear shrinkage of the sample
covariance matrix of daily log-returns (250 days when $p=30, 40$ and $50$). }
}
\end{flushleft}
\end{table}

\begin{table}[tbp]
\caption{The out-of-sample performance of different daily rebalanced
strategies for the Markowitz portfolio with momentum signal between August
28, 2013 to December 31, 2013.}\centering
\begin{tabular}{|c||c|c|c|c|c|c|c|}
\hline
\multicolumn{8}{|c|}{{\small {Period: 08/28/2013---12/31/2013}}} \\ \hline
$p=30$ & EW-TQ & SP & SPo & LS & LSo & SQrD & SQrM \\ \hline
AV & 54.79 & 10.59 & 9.79 & 29.02 & 29.02 & 39.95 & 31.35 \\ \hline
SD & 11.82 & 11.14 & 11.06 & 10.53 & 10.53 & 10.69 & {$10.02^*$} \\ \hline
IR & \textbf{4.64 } & 0.95 & 0.89 & 2.76 & 2.76 & {3.74} & {\ 3.11} \\
\hline\hline
$p=40$ & EW-TQ & SP & SPo & LS & LSo & SQrD & SQrM \\ \hline
AV & 63.78 & 25.11 & 27.28 & 41.99 & 41.99 & 45.21 & 43.11 \\ \hline
SD & 12.73 & 10.84 & 10.97 & 11.00 & 11.00 & 10.92 & $9.55^*$ \\ \hline
IR & \textbf{5.01} & 2.32 & 2.49 & 3.82 & 3.82 & 4.14 & {4.51} \\
\hline\hline
$p=50$ & EW-TQ & SP & SPo & LS & LSo & SQrD & SQrM \\ \hline
AV & {53.77} & 23.20 & 20.94 & 37.87 & 37.87 & 38.31 & 42.54 \\ \hline
SD & 11.80 & 9.76 & 11.13 & 10.16 & 10.16 & 10.14 & {$9.05^*$} \\ \hline
IR & 4.56 & {2.38} & 1.88 & 3.73 & 3.73 & 3.78 & \textbf{4.70} \\ \hline
\end{tabular}%
\par
\begin{flushleft}
{\footnotesize {Note: AV, SD, IR denote the average, standard deviation, and
information ratio of 87 daily log-returns respectively. AV, SD are
annualized and in percent. The smallest numbers in the row labeled by SD is
reported in bold face. SPo corresponds to the case where $%
\mbox{\boldmath
$\Sigma$}_{T-h,T}$ is estimated by the sample covariance matrix of daily
log-returns (190, 230 and 130 days when $p=30, 40$ and $50$, respectively).
LSo corresponds to the case where $\mbox{\boldmath
$\Sigma$}_{T-h,T}$ is estimated by the linear shrinkage of the sample
covariance matrix of daily log-returns (250 days when $p=30, 40$ and $50$). }
}
\end{flushleft}
\end{table}

\begin{figure}[tbp]
\caption{Information ratios and standard deviations of log-returns of four
strategies based on rolling windows of historical data for the GMV and MwM
portfolios when $p=30$. }\centering
\includegraphics[width=7.5cm]{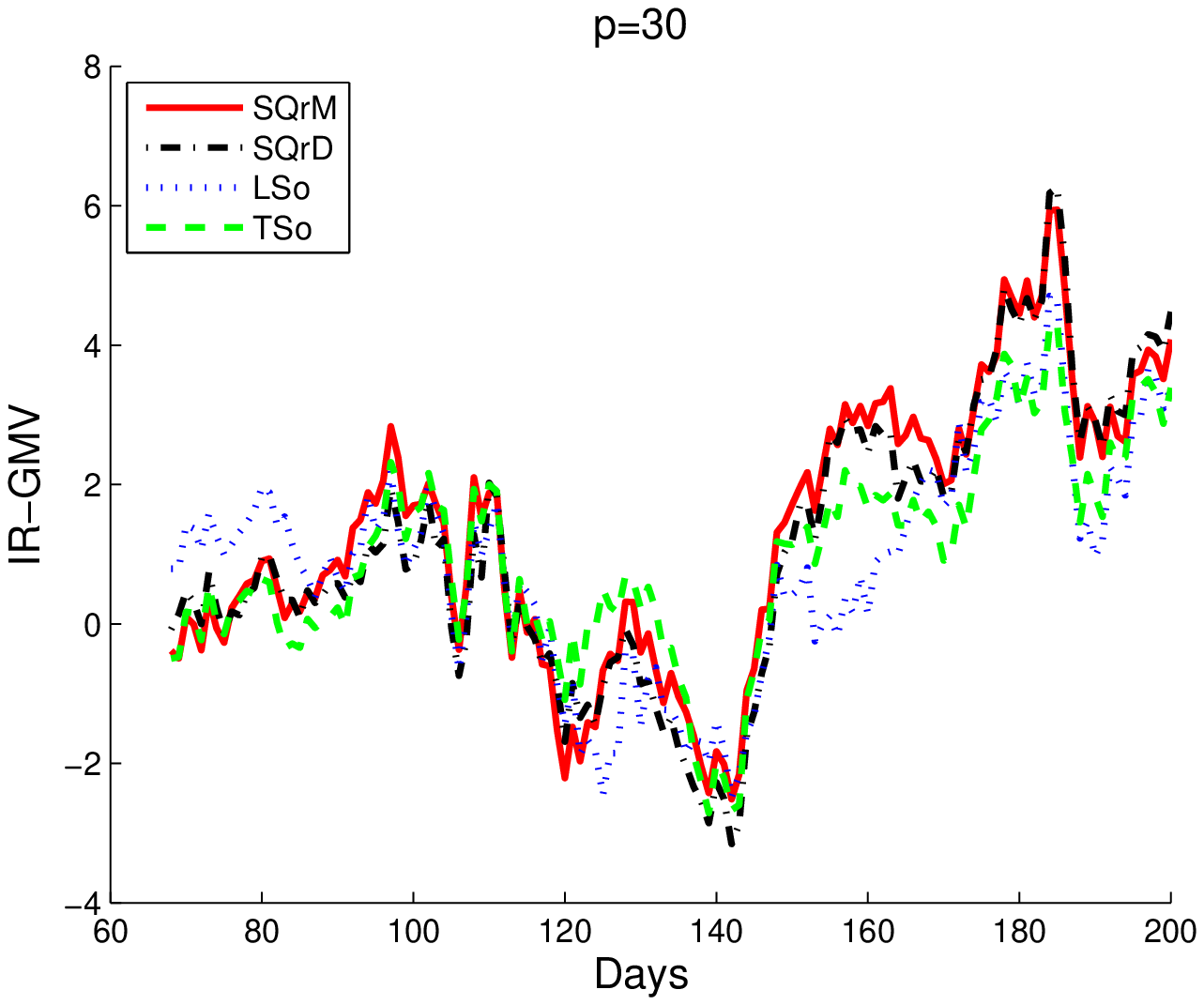} %
\includegraphics[width=7.5cm]{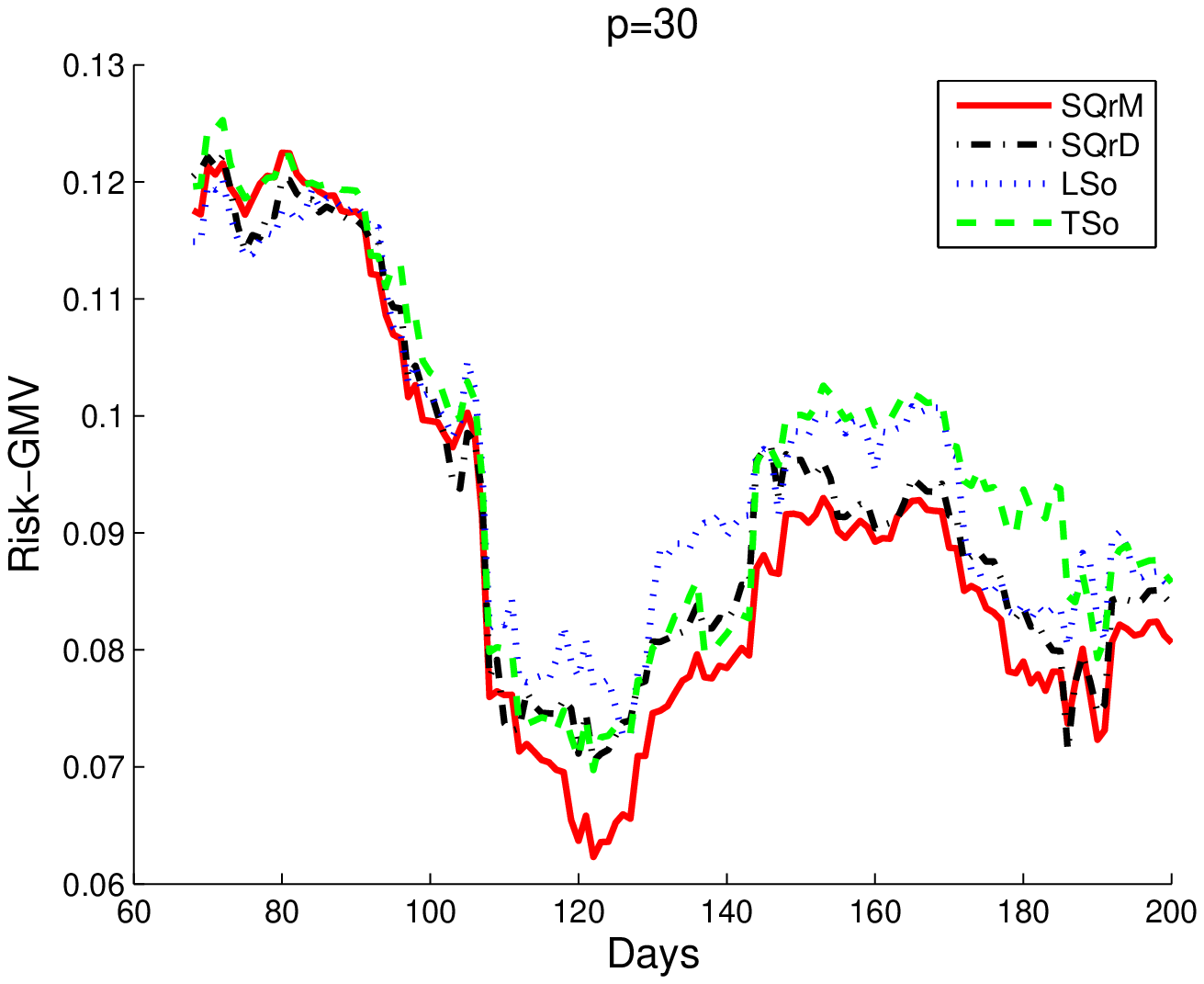} %
\includegraphics[width=7.5cm]{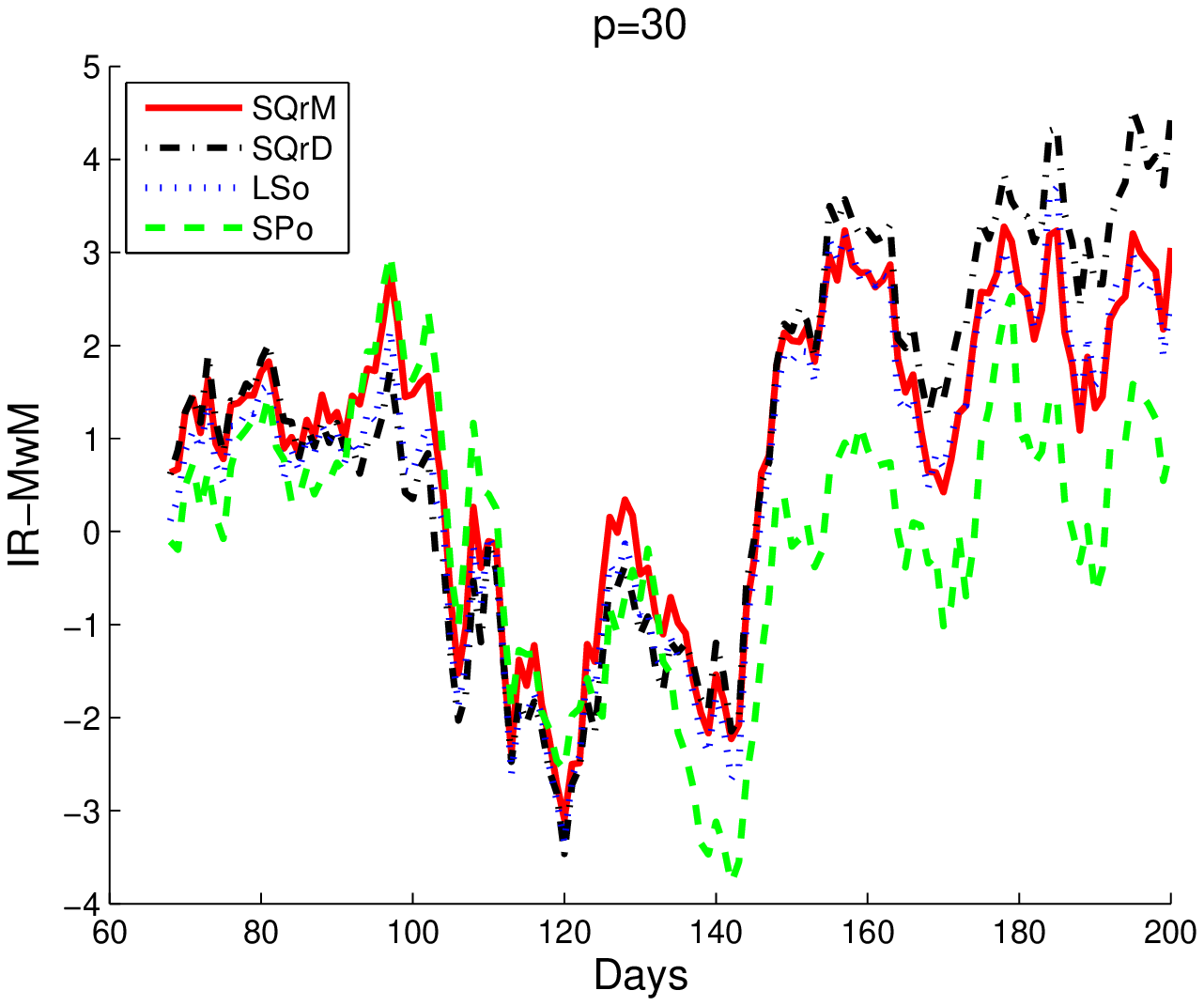} %
\includegraphics[width=7.5cm]{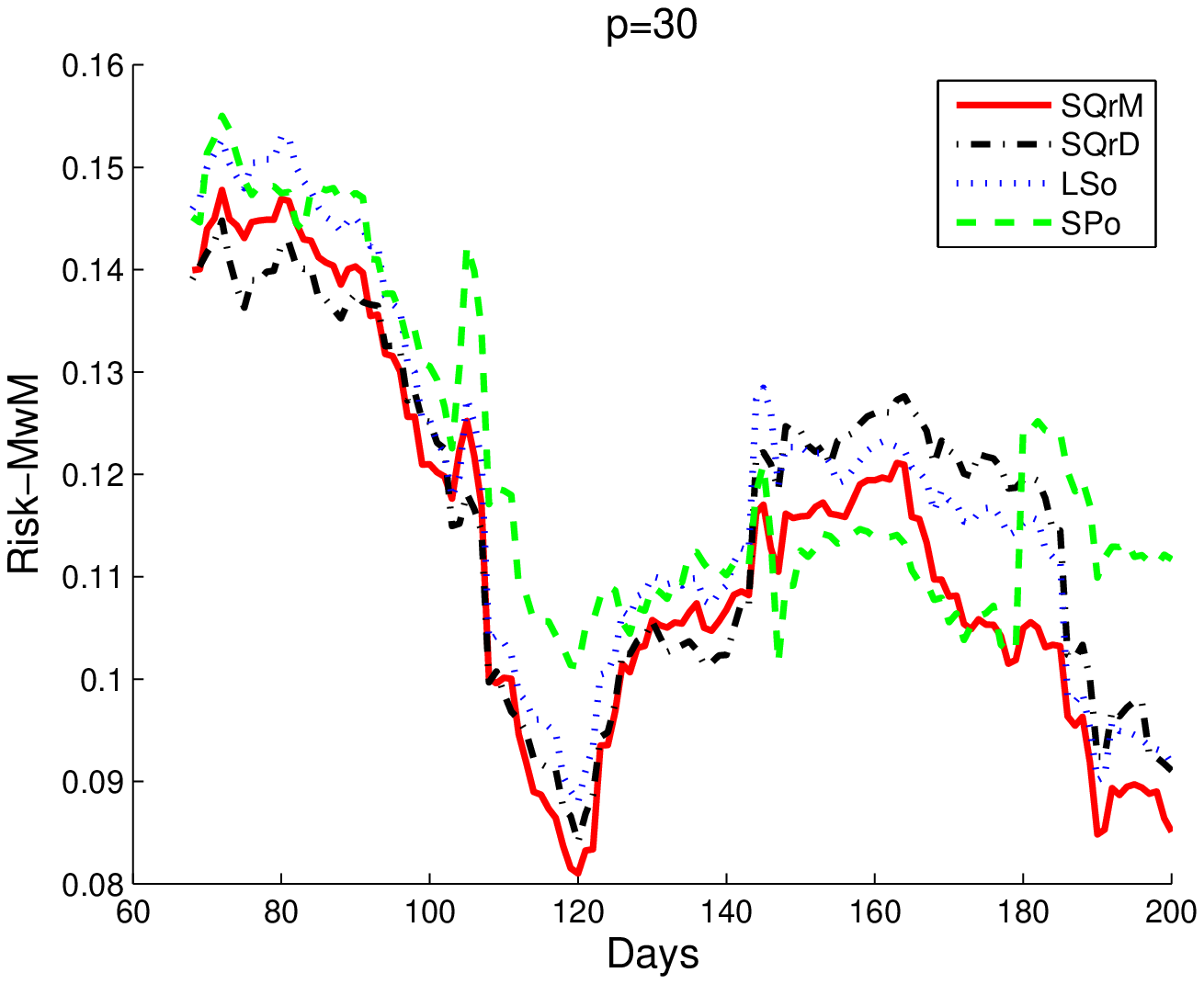}
\par
\begin{flushleft}
{\footnotesize {Note: Rolling windows of annualized standard deviations and
information ratios of log-returns for the GMV and MwM portfolios. Each point
is the standard deviation or information ratio of 42 log-returns of each
portfolio strategy. Move one trading day forward at one time such that there
are 133 different investment periods and each period contains 42 days (two
months). The upper plots correspond to: SQrM uses 1 day of all intra-day
data and 9 days of 15-minute data; SQrD uses 5 days of all intra-day data
and 110 days of daily data; LSo uses 250 daily data; TSo uses 10 days of
intra-day data. The bottom plots correspond to: SQrM use 5 days of all
intra-day data and 14 days of 15-minute data; SQrD uses 2 days of all
intra-day data and 90 days of daily data; LSo and SPo use 110 days of daily
data. } }
\end{flushleft}
\end{figure}

\begin{figure}[tbp]
\caption{Information ratios and standard deviations of log-returns of four
strategies based on rolling windows of historical data for the GMV and MwM
portfolios when $p=40$. }\centering
\includegraphics[width=7.5cm]{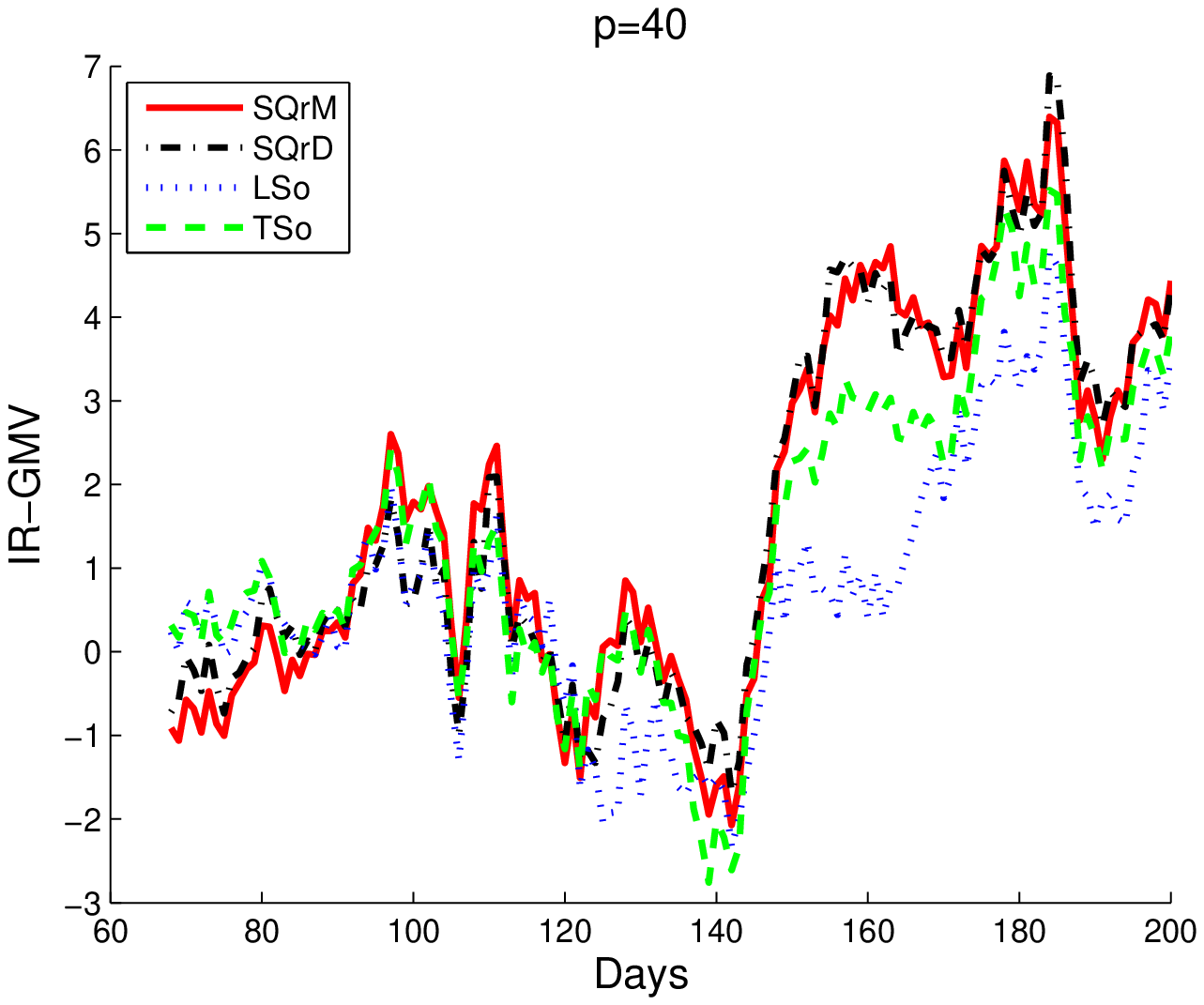} %
\includegraphics[width=7.5cm]{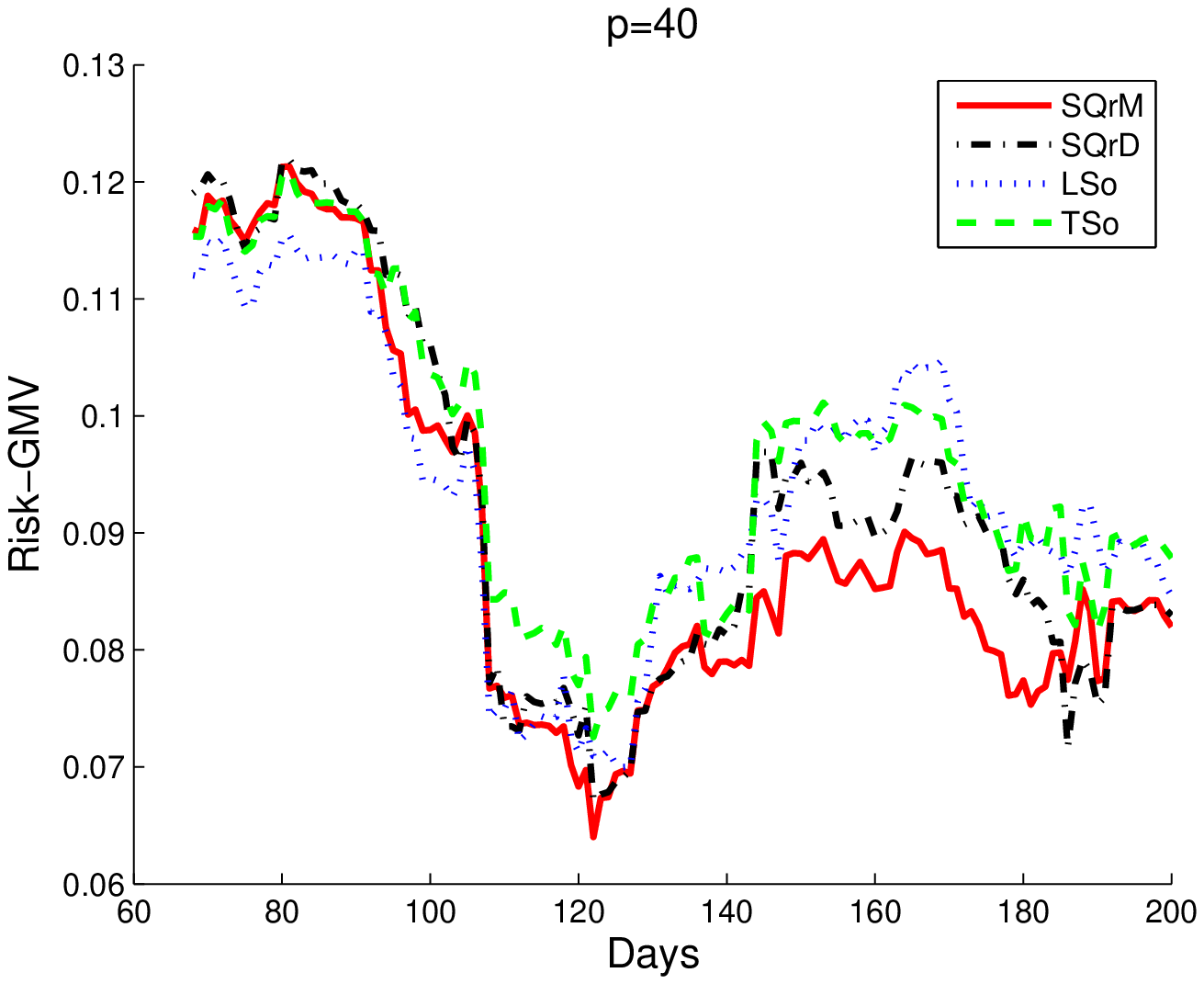} %
\includegraphics[width=7.5cm]{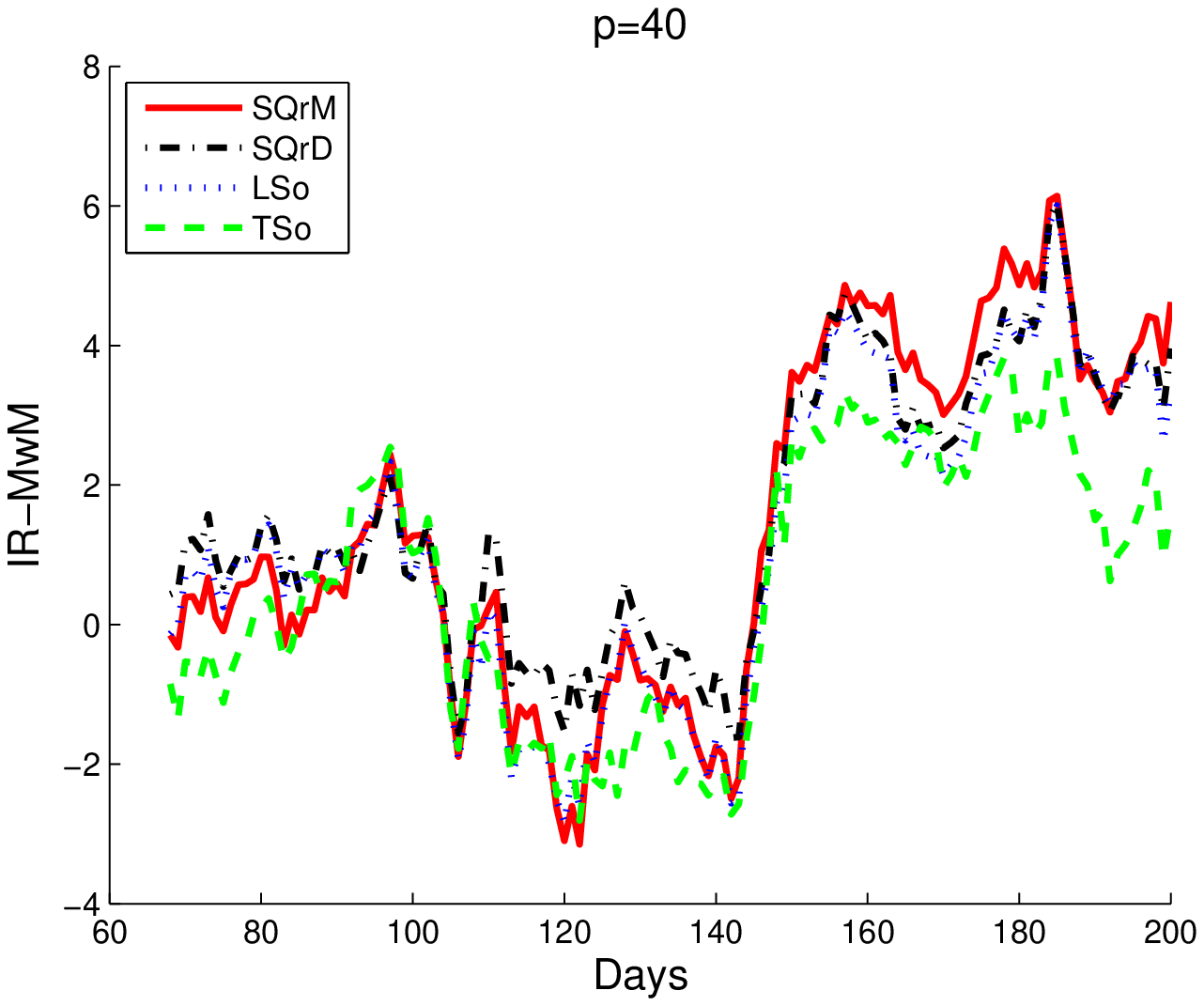} %
\includegraphics[width=7.5cm]{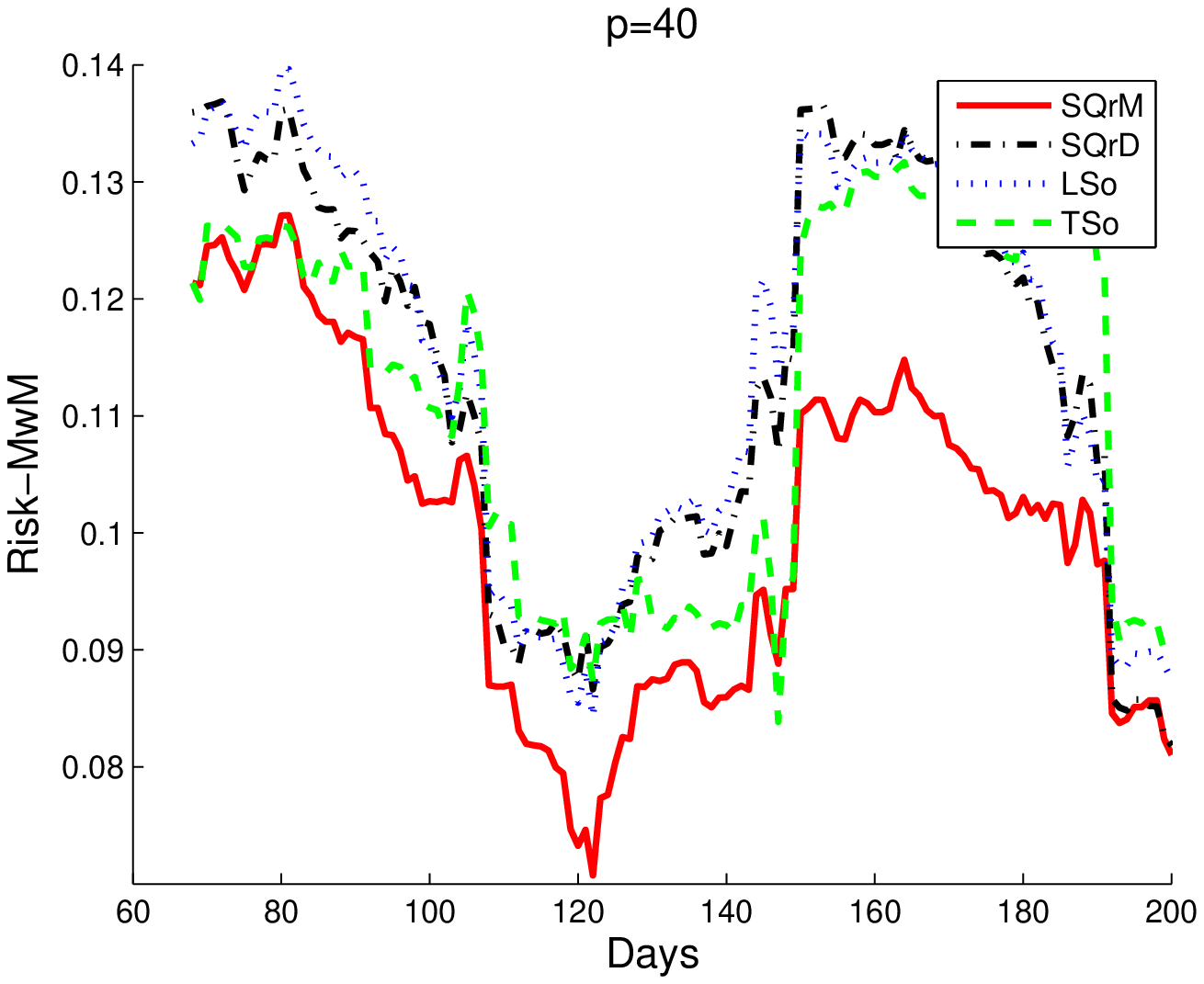}
\par
\begin{flushleft}
{\footnotesize {Note: Rolling windows of annualized standard deviations and
information ratios of log-returns for the GMV and MwM portfolios. Each point
is the standard deviation or information ratio of 42 log-returns of each
portfolio strategy. Move one trading day forward at one time such that there
are 133 different investment periods and each period contains 42 days (two
months). The upper plots correspond to: SQrM uses 1 day of all intra-day
data and 17 days of 15-minute data; SQrD uses 1 days of all intra-day data
and 110 days of daily data; LSo uses 250 daily data; TSo uses 10 days of
intra-day data. The bottom plots correspond to: SQrM use 4 days of all
intra-day data and 15 days of 15-minute data; SQrD uses 5 days of all
intra-day data and 200 days of daily data; LSo and SPo use 250 and 190 days
of daily data, respectively.} }
\end{flushleft}
\end{figure}

\begin{figure}[tbp]
\caption{Information ratios and standard deviations of log-returns of four
strategies based on rolling window of historical data for the GMV and MwM
portfolios when $p=50$. }\centering
\includegraphics[width=7.5cm]{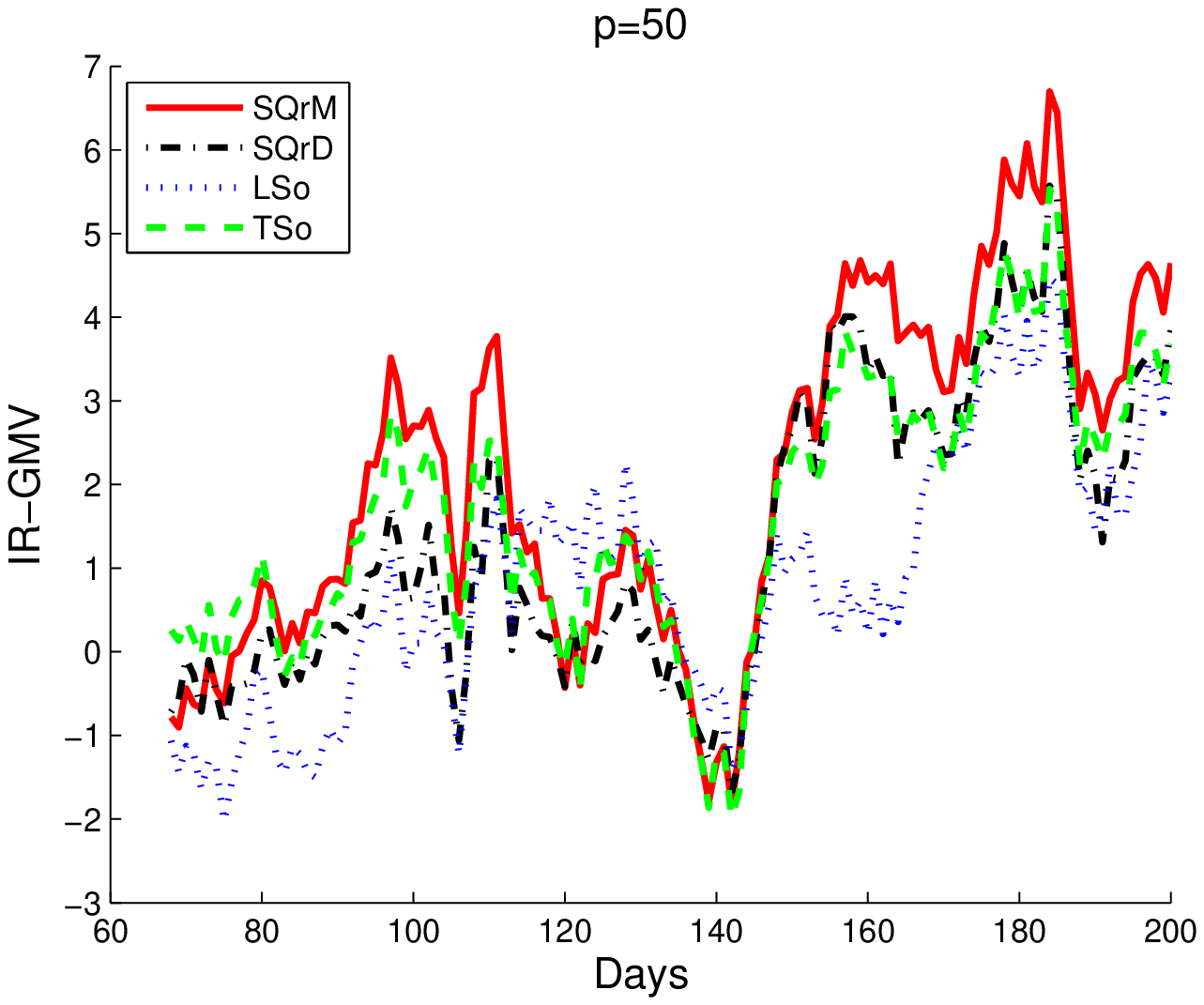} %
\includegraphics[width=7.5cm]{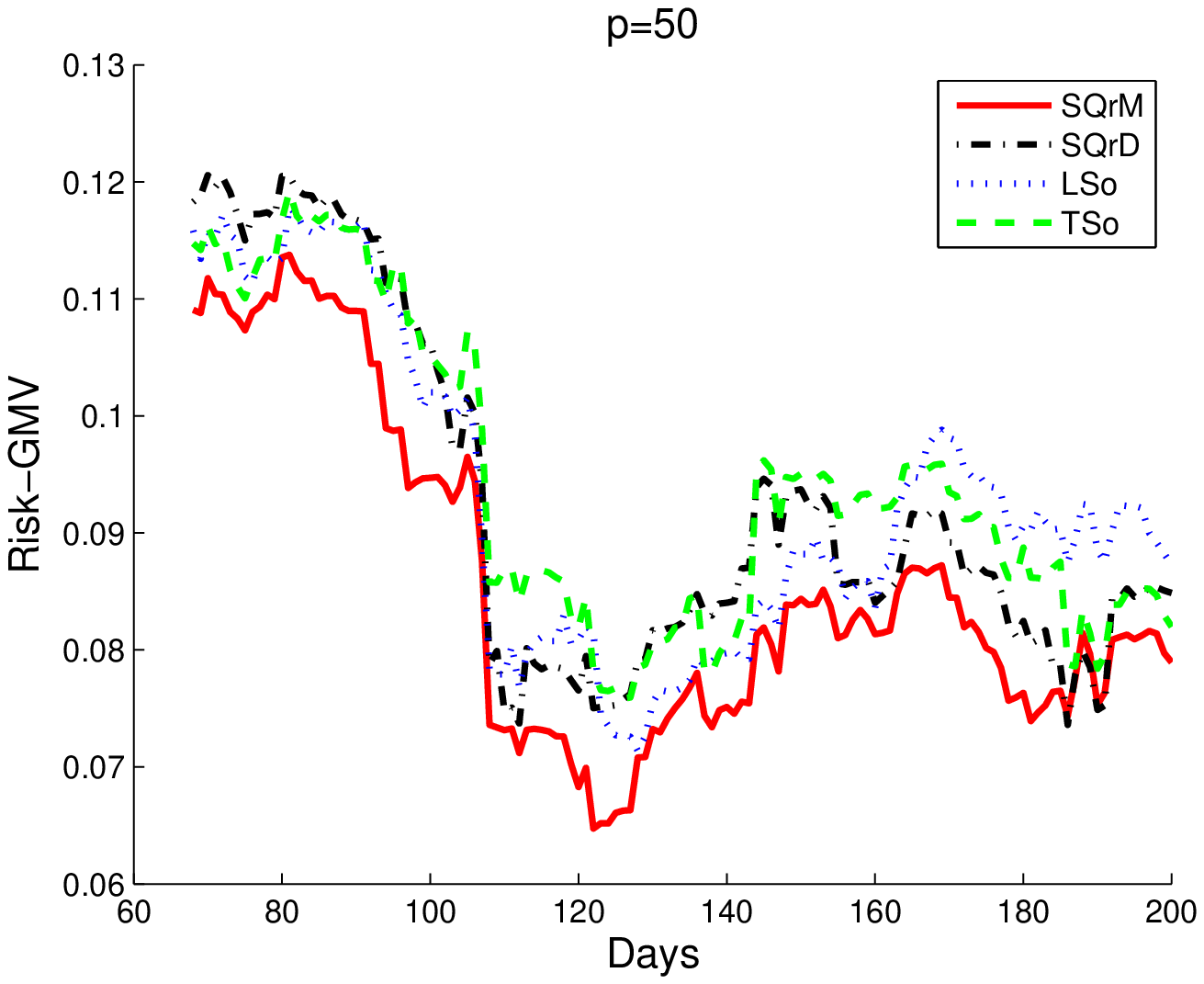} %
\includegraphics[width=7.5cm]{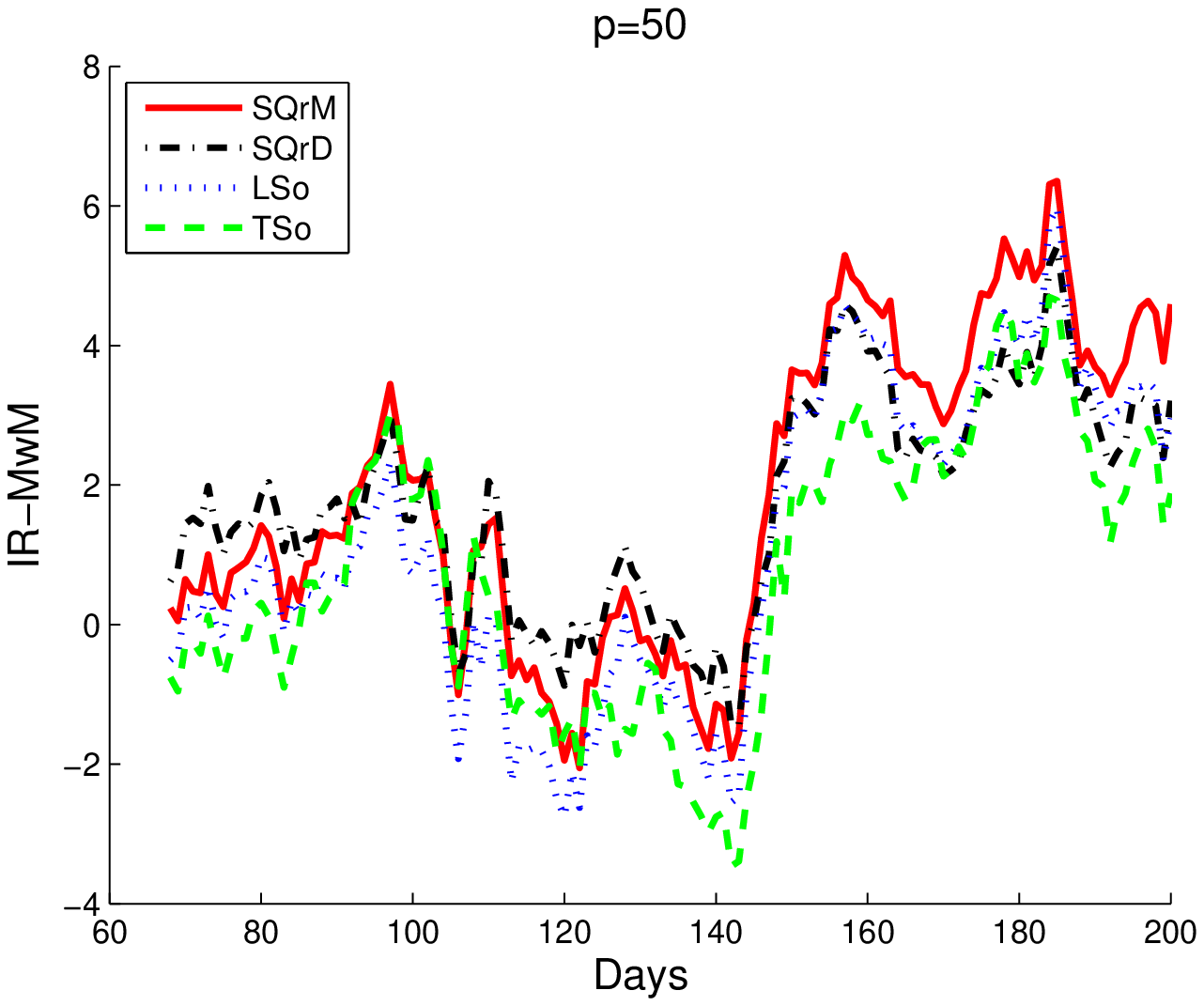} %
\includegraphics[width=7.5cm]{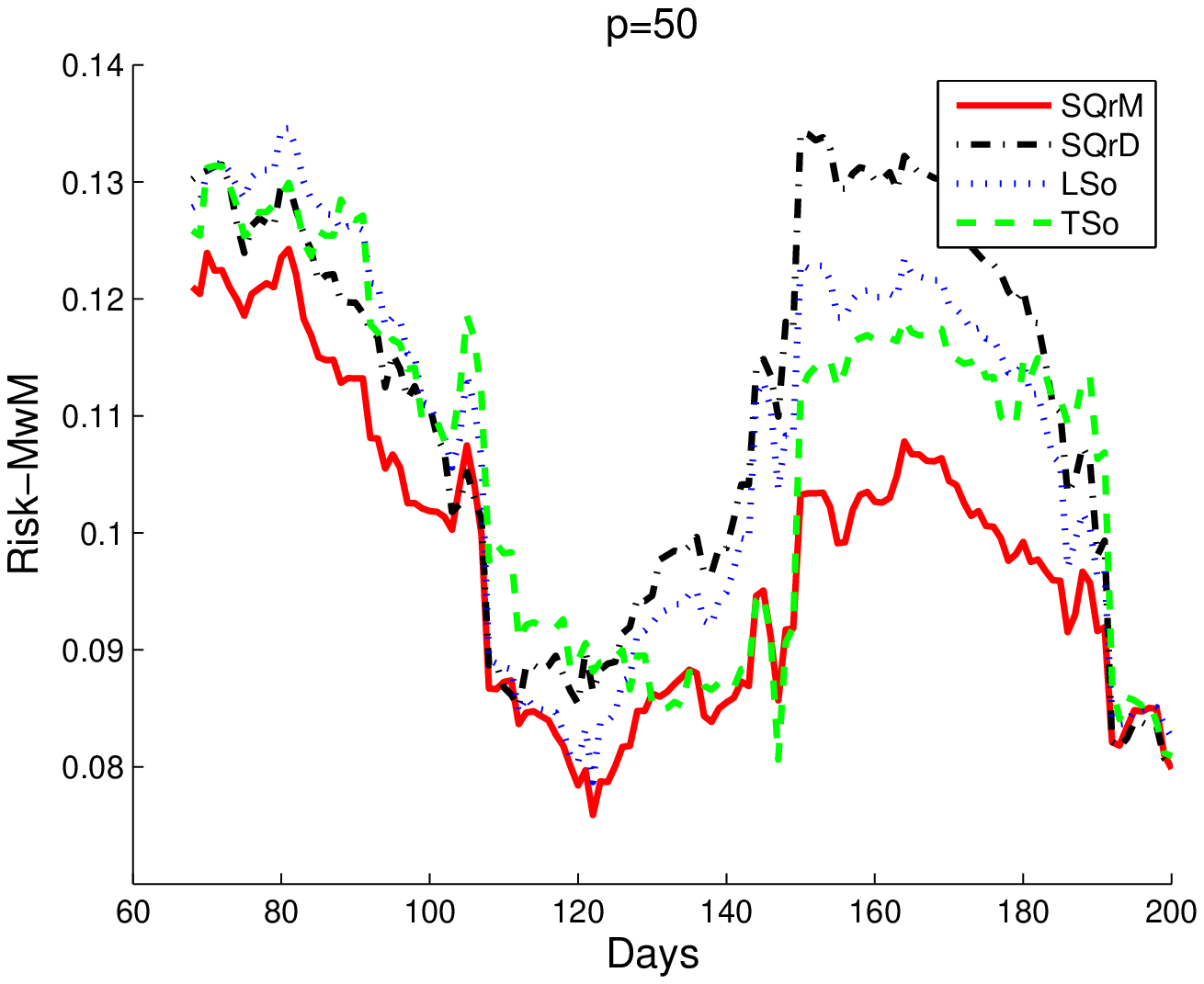}
\par
\begin{flushleft}
{\footnotesize {Note: Rolling windows of annualized standard deviations and
information ratios of log-returns for the GMV and MwM portfolios. Each point
is the standard deviation or the information ratio of 42 log-returns of each
portfolio strategy. Move one trading day forward at one time such that there
are 133 different investment periods and each period contains 42 days (two
months). The upper plots correspond to: SQrM uses 1 day of all intra-day
data and 15 days of 15-minute data; SQrD uses 1 days of all intra-day data
and 130 days of daily data; LSo uses 250 daily data; TSo uses 9 days of
intra-day data. The bottom plots correspond to: SQrM use 4 days of all
intra-day data and 13 days of 15-minute data; SQrD uses 5 days of all
intra-day data and 90 days of daily data; LSo and SPo use 250 and 130 days
of daily data, respectively. }}
\end{flushleft}
\end{figure}

\begin{figure}[tbp]
\caption{Standard deviation of log-returns of SQrD and LS in 174 investment
days for the GMV when $p=30$ and $p=50$.}\centering
\par
\includegraphics[width=7.5cm]{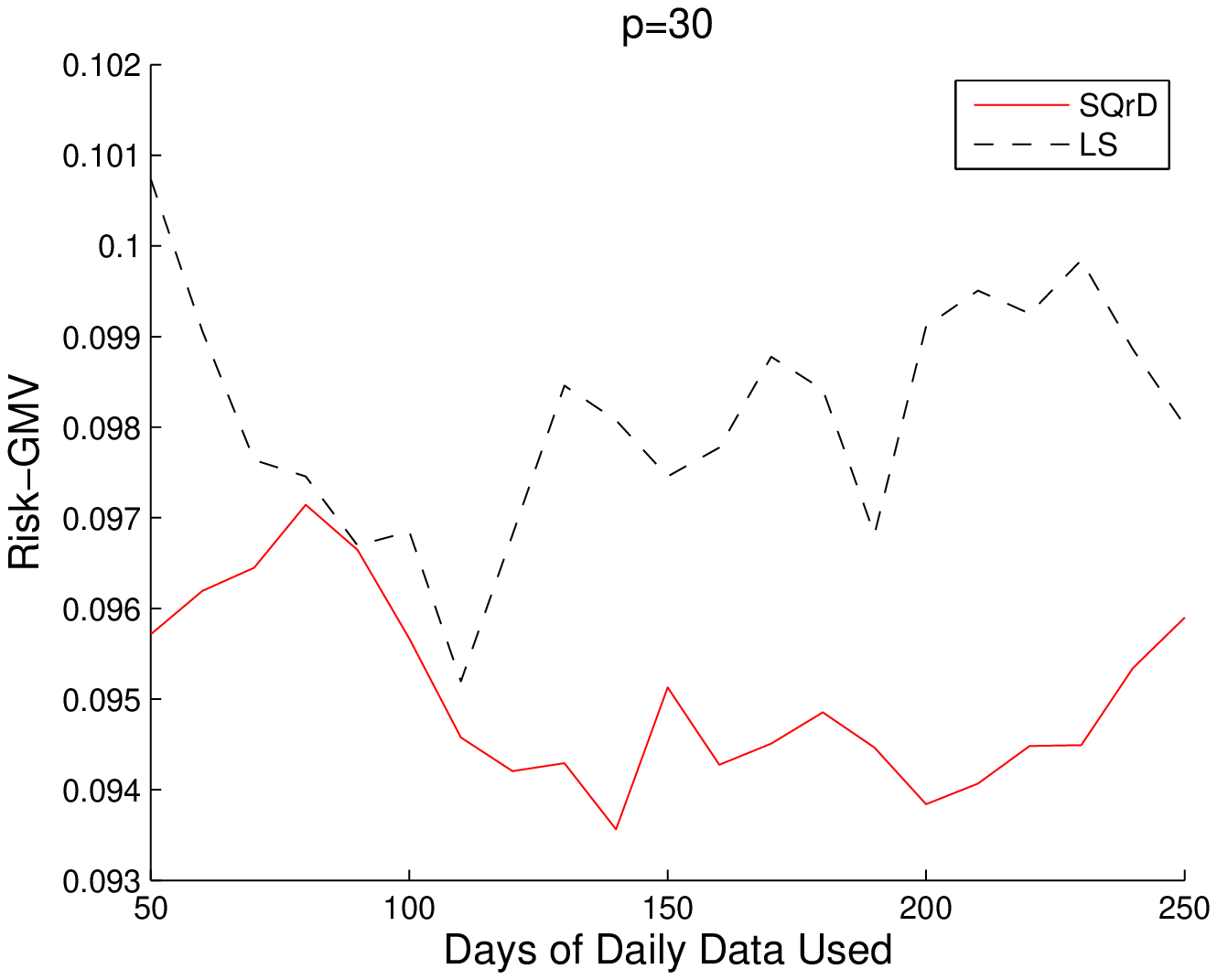} \includegraphics[width=7.5cm]{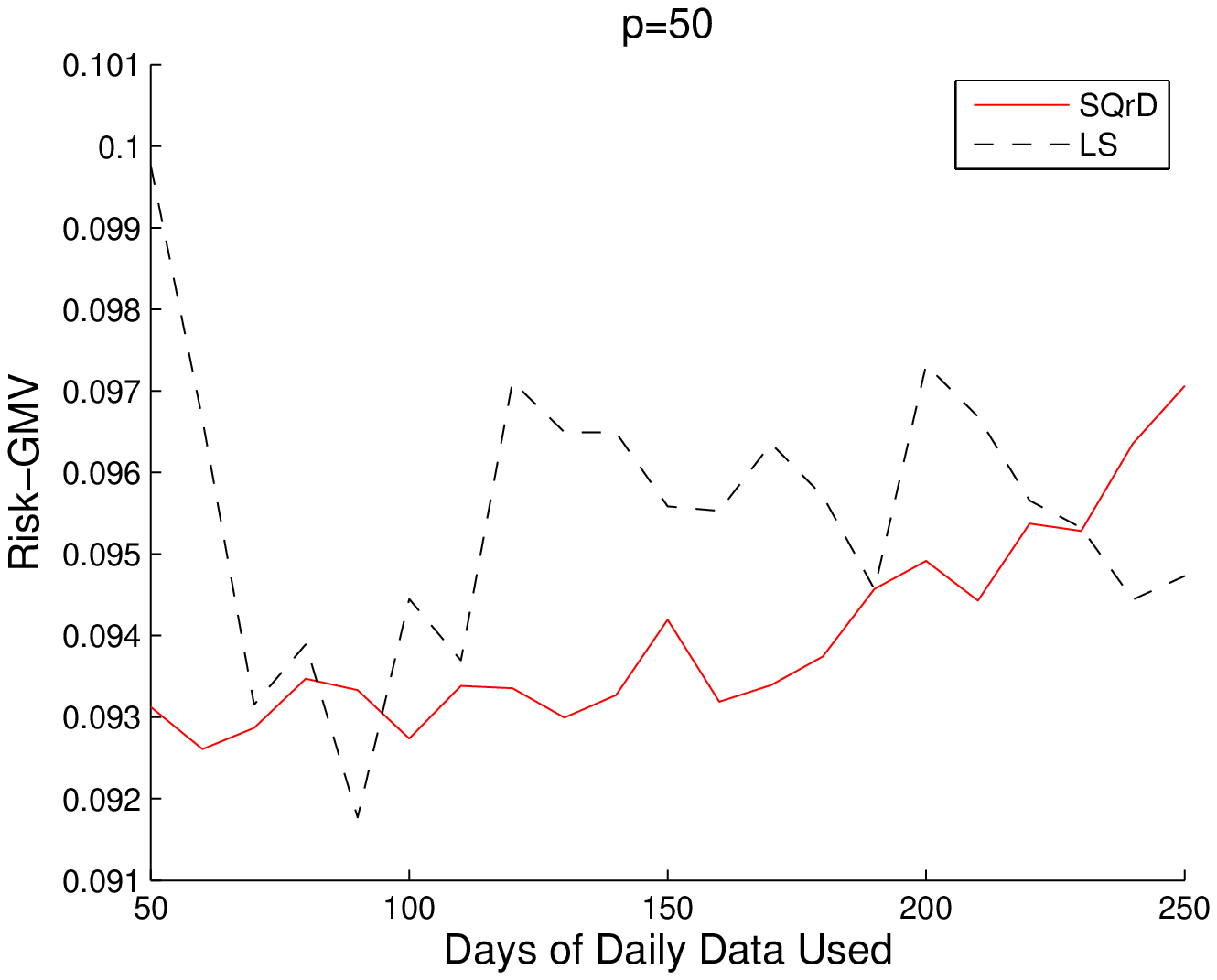}
\par
\begin{flushleft}
{\footnotesize {Note: Each point is the standard deviation of 174
log-returns of each portfolio strategy. } }
\end{flushleft}
\end{figure}

\end{document}